\documentclass[12pt]{amsart}
\usepackage{amssymb}
\usepackage{url}
\usepackage{graphicx}
\usepackage{tikz-cd}
\usepackage{mathtools}
\usepackage{upgreek}
\usepackage{comment}
\usepackage{stmaryrd}
\usepackage{ulem}
\usepackage[all]{xy}

\allowdisplaybreaks

\newtheorem{theorem}[equation]{Theorem}

\newtheorem{lemma}[equation]{Lemma}
\newtheorem{proposition}[equation]{Proposition}
\newtheorem{corollary}[equation]{Corollary}

\newtheorem*{theorem:derhamisomorphism}{Theorem~\ref{T:derhamisomorphism}}
\newtheorem*{theorem:characterization}{Theorem~\ref{T:characterization}}

\theoremstyle{definition}
\newtheorem{definition}[equation]{Definition}
\newtheorem{example}[equation]{Example}

\newtheorem{remark}[equation]{Remark}

\numberwithin{equation}{section}

\newcommand{\TT}{\mathbb{T}}

\newcommand{\CC}{\mathbb{C}}

\newcommand{\zz}{\mathbf{z}}
\newcommand{\tz}{\widetilde{\mathbf{z}}}

\newcommand{\dbl}{\llbracket}
\newcommand{\dbr}{\rrbracket}

\newcommand{\bb}{\mathbf{b}}

\DeclareMathOperator{\Ker}{Ker}
\DeclareMathOperator{\GL}{GL}
\DeclareMathOperator{\Mat}{Mat}

\DeclareMathOperator{\Eis}{Eis}

\DeclareMathOperator{\Res}{Res}
\DeclareMathOperator{\Hol}{Hol}

\DeclareMathOperator{\Id}{Id}

\DeclareMathOperator{\Cof}{Cof}

\newcommand{\tr}{\mathrm{tr}}

\newcommand{\inorm}[1]{{\lvert #1 \rvert}}

\begin{document}

\title[Drinfeld modular forms]{A family of  Hecke eigenforms for Drinfeld modular forms of arbitrary rank}

\author{O\u{g}uz Gezm\.{i}\c{s}}
\address{O. Gezm\.{i}\c{s}\\Department of Mathematics, National Tsing Hua University, Hsinchu City 30042, Taiwan R.O.C.}
\email{gezmis@math.nthu.edu.tw}

\author{\"{O}zge \"{U}lkem}
\address{\"{O}. \"{U}lkem\\ Institute of Mathematics, Academia Sinica, Astronomy Mathematics Building, No. 1, Roosevelt Rd. Sec. 4, Taipei 10617, Taiwan R.O.C.}

\email{ozgeulkem@gmail.com}

\date{\today}



\begin{abstract}  We aim to provide an infinite family of cuspidal Drinfeld-Hecke eigenforms given in terms of a  determinant of twisted Eisenstein series. Our main tool is the theory of vectorial Drinfeld modular forms, previously introduced by Pellarin \cite{Pel12} and extensively studied by Pellarin \cite{Pel19} as well as in his joint work with Perkins \cite{PP18} in the rank two setting. We will develop this theory in the present paper for the arbitrary rank setting by using a particular representation.
\end{abstract}
\maketitle
\section{Introduction}
 Let $\mathbb{F}_q$ be the finite field with $q$ elements. We let $A$ be the polynomial ring, denoted by $\mathbb{F}_q[\theta]$ where $\theta$ is an indeterminate over $\mathbb{F}_q$. We further let $K$ be its fraction field. We consider the norm $|\cdot|$ at the infinite place normalized so that $|\theta|=q$. We let $K_{\infty}$ be the completion of $K$ with respect to $|\cdot|$ which may be given as the formal Laurent series $\mathbb{F}_q(\!(1/\theta)\!)$. We also set $\mathbb{C}_{\infty}$ to be the completion of a fixed algebraic closure of $K_{\infty}$. To describe our objects, letting $t$ be an indeterminate over $\mathbb{C}_{\infty}$, we also consider the Tate algebra $\TT$ given by
\[
\TT:=\left\{\sum_{n=0}^{\infty}a_nt^n\in \mathbb{C}_{\infty} \dbl t \dbr \ \ | \ \ |a_n|\to 0 \text{ as } n\to \infty\right\}.
\]
One can define \textit{the Gauss norm} $||\cdot||$ on $\mathbb{T}$ defined by 
\[
||g||:=\max\{|a_n|\ \ | \ \ n\geq 0\}, \ \ g=\sum_{n=0}^{\infty}a_nt^n\in \mathbb{T}.\]
Note that $(\mathbb{T},||\cdot||)$ forms a Banach space. Moreover, $\mathbb{T}$ may be seen as the algebra of holomorphic functions that converges on the closed unit disk in $\mathbb{C}_{\infty}$. 
For any $\ell \in \mathbb{Z}$, we define the automorphism $(\cdot)^{(\ell)}:\mathbb{T}\to \mathbb{T}$ by 
\begin{equation}\label{E:twisting}
g^{(\ell)}:=\sum_{n=0}^{\infty}a_n^{q^{\ell}}t^n\in \mathbb{T}, \ \ g=\sum_{n=0}^{\infty}a_nt^n\in \mathbb{T}.
\end{equation}
We further extend the map $(\cdot)^{(\ell)}$ to the elements of  $ \Mat_{m\times n}(\mathbb{T}) $ and for any $B=(b_{ij})_{i,j}\in \Mat_{m\times n}(\mathbb{T})$, we define $B^{(\ell)}:=(b_{ij}^{(\ell)})_{i,j}\in \Mat_{m\times n }(\mathbb{T})$.

 Let $r\geq 2$ and consider \textit{the Drinfeld's upper half plane} given by 
\[
\Omega^r:=\mathbb{P}^{r-1}(\CC_{\infty})\setminus \{K_{\infty}\text{-rational hyperplanes}\}.
\]
We refer the reader to \S2.2 for more details on its rigid analytic space structure. Throughout this paper, we identify elements of $\Omega^r$ with the tuple $(z_1,\dots,z_r)^{\tr}\in \mathbb{C}_{\infty}^r$ whose entries are $K_{\infty}$-linearly independent and normalized so that $z_r=1$.  We further note that when $r=2$, we may indeed identify $\Omega^2$ set theoretically with $\mathbb{C}_{\infty}\setminus K_{\infty}$. 

Drinfeld modular forms were defined by D. Goss in this PhD thesis (see also \cite{Gos80a}) analytically as $\mathbb{C}_{\infty}$-valued rigid analytic functions on $\Omega^2$ satisfying a certain functional equation under the action of $\GL_2(A)$ on $\Omega^2$ and a growth condition at cusps. Algebraically, these objects may be given as global sections of an ample invertible sheaf on the compactified Drinfeld moduli space. He further studied the Hecke operators acting on them and showed that the space of Drinfeld modular forms are stable under these operators and cuspidality is also preserved \cite[Thm. 3.14]{Gos80a}. Moreover, he provided examples for eigenforms such as Eisenstein series and the discriminant function of rank two \cite{Gos80b}. Later on, Gekeler studied these operators in \cite{Gek88} and provided more examples of eigenforms such as \textit{Gekeler's $h$-function} $h_2$ (see Example \ref{Ex:1} for more details on the discriminant function and the function $h_2$). In another direction, generalizing the work of Lopez \cite{Lop10}, Petrov \cite{Pet13} showed that if a Hecke eigenform has \textit{an $A$-expansion} then its eigenvalues can be explicitly described (see also Remark \ref{R:intr}).

The theory of Drinfeld modular forms are not only interesting in the rank two setting but also can be generalized to the arbitrary rank case. In \cite{BBP18}, based on the Satake compactification introduced by Pink in \cite{Pin13}, Basson, Breuer and Pink introduced Drinfeld modular forms of arbitrary rank analytically as $\mathbb{C}_{\infty}$-valued rigid analytic functions on $\Omega^r$ satisfying a certain functional equation under the action of $\GL_r(A)$ on $\Omega^r$ and a growth condition at cusps (see also \cite{Bas14}). Similar to the rank two case, algebraically, they can be given as global sections of an ample invertible sheaf on the compactified Drinfeld moduli space. Moreover, they introduced the Hecke operators acting on Drinfeld modular forms and showed that Eisenstein series are Hecke eigenforms \cite[Cor. 14.12]{BBP18}. Later on, Basson \cite{Bas23} analyzed these operators in a special case and determined their effect on the $u$-expansion of Drinfeld modular forms \cite[Thm. 3.6, Cor. 3.11]{Bas23}. Furthermore, he provided examples of Hecke eigenforms such as the discriminant function of arbitrary rank and coefficient forms \cite[Thm. 4.1, Thm. 4.2]{Bas23}. We also refer the reader to the recent work  \cite{Gek26} of Gekeler where he studies Hecke operators constructed by using $A$-lattices in $\mathbb{C}_{\infty}$ (see \S2.1 for the explicit definition of $A$-lattices).

Our goal in the present paper is to introduce more examples of Hecke eigenforms in the arbitrary rank setting which may be written as a determinant of twisted Eisenstein series. To state our main result, we first introduce some fundamental objects. For any $a\in A$, we set $a(t):=a_{|\theta=t}$. For each $\mathbf{z}=(z_1,\dots,z_r)^{\tr}\in \Omega^r$, $1\leq i \leq r$ and $k=q^j\in \mathbb{Z}_{\geq 1}$ with $j\geq 0$, we define  
\[
\mathcal{E}^{[i]}_{k}(\zz,t):=\sum_{\substack{{a_1,\dots,a_r\in A}\\{a_1,\dots,a_r \text{ not all zero}}}}\frac{a_i(t)}{(a_1z_1+\dots+a_rz_r)^k}\in \mathbb{T}.
\] 
These functions, in the rank two setting, were first introduced by Pellarin in \cite{Pel12} and then extensively studied by Pellarin \cite{Pel19} and Pellarin and Perkins \cite{PP18}. 

Consider the $\mathbb{T}$-valued function $\mathcal{H}_r(\cdot,t):\Omega^r\to \mathbb{T}$ given by 
\[
\mathcal{H}_r(\zz,t):=\det \begin{pmatrix}
    \mathcal{E}^{[1]}_{1}(\zz,t) & \cdots &  \mathcal{E}^{[r-1]}_{1}(\zz,t)\\
\vdots & & \vdots \\
\mathcal{E}^{[1]}_{q^{r-2}}(\zz,t) & \cdots &  \mathcal{E}^{[r-1]}_{q^{r-2}}(\zz,t)
\end{pmatrix},\ \ \zz\in \Omega^r.
\]
Letting $(-\theta)^{1/(q-1)}$ be a fixed $(q-1)$-st root of $-\theta$, we also consider \textit{the Carlitz period} $\tilde{\pi}\in \mathbb{C}_{\infty}^{\times}$ which is defined by the infinite product
\begin{equation}\label{E:pi}      \tilde{\pi}:=\theta(-\theta)^{1/(q-1)}\prod_{i=1}^\infty\left(1-\theta^{1-q^i}\right)^{-1}\in\mathbb{C}_\infty^\times.
\end{equation}
Our main result (it will be restated as Theorem \ref{T:funcH} and Corollary \ref{C:main2} later) can be described as follows.
\begin{theorem}\label{T:main} The following statements hold.
\begin{itemize}
\item[(i)]
For any fixed $\zz\in \Omega^r$, $\mathcal{H}_r(\zz,\cdot )$ may be extended to an entire function of the variable $t$. Moreover, for any $n\geq r-1$,  letting $D_0:=1$ and $D_{\ell}:=(\theta^{q^{\ell}}-\theta)D_{\ell-1}\in A$ for each $\ell\geq 1$, we have
\[
\mathcal{H}_r(\zz,\theta^{q^n})=\tilde{\pi}^{\frac{q^r-1}{q-1}-q^{n}}D_{n-(r-1)}^{q^{r-1}}h_r(\zz)\alpha_{n-(r-1)}(\zz)^{q^{r-1}}
\]
where $h_r$ is \textit{Gekeler's $h$-function} and $\alpha_{n-(r-1)}(\zz)$ is the $n-(r-1)$-th coefficient of the exponential series of the Drinfeld module $\phi^{\zz}$ corresponding to the $A$-lattice $A^r\zz$. In particular, we have
\[
\mathcal{H}_r(\zz,\theta^{q^{r-1}})=\tilde{\pi}^{\frac{q^{r-1}-1}{q-1}}h_r(\zz).
\]
Furthermore, $ \mathcal{H}_r(\zz,\theta^{q^n}) $ is a Drinfeld cusp form of weight $ \frac{q^{r-1}-1}{q-1}+q^n $ and type $1$ for $\GL_r(A)$.
\item[(ii)] Let $\mathfrak{p}$ be a monic irreducible polynomial in $A$ and $T_{\mathfrak{p},r}$ be the Hecke operator acting on the space of Drinfeld modular forms of rank $r$. Then for each $n\geq r-1$,  $\mathcal{H}_r(\zz,\theta^{q^n})$ is a Hecke eigenform and
\[
T_{\mathfrak{p},r}(\mathcal{H}_r(\zz,\theta^{q^n}))=\mathfrak{p}^{1+q+\dots+q^{r-2}}\mathcal{H}_r(\zz,\theta^{q^n}).
\]
\end{itemize}
    
\end{theorem}

Our strategy to prove Theorem \ref{T:main} can be described as follows. First, in \S3.2, we analyze a fundamental identity \eqref{E:mainid} shown by Chen and the first author in \cite{CG22}. More precisely, we focus on a certain entry of this identity that allows us to relate $\mathcal{H}_r$ to the $(r-1)$-st twist of the ratio of an Anderson generating function with the Anderson-Thakur function (Theorem \ref{T:funcH}(i)). Investigating the removable singularity of this ratio at $t=\theta^{q^n}$ for $n\geq r-1$ then implies the first part of the theorem. To prove the second part, using a particular representation, we introduce vectorial Drinfeld modular forms (abbreviated as VDMFs later) in the sense of Pellarin \cite{Pel19} and Pellarin and Perkins \cite{PP18}. In \S3.3, we provide explicit examples $\mathcal{G}_1,\dots,\mathcal{G}_r$ for these objects so that the last entry of $\mathcal{G}_r$ is equal to $\mathcal{H}_r$. Moreover, in Theorem \ref{T:str}, we show that these functions indeed generate the $\mathbb{T}$-module of VDMFs over the $\mathbb{C}_{\infty}$-vector spaces of Drinfeld modular forms of certain weight and type. This will imply that the space of VDMFs of weight $\frac{q^{r-1}-1}{q-1}$ with respect to our chosen representation is generated by $\mathcal{G}_r$ over $\mathbb{T}$. Later on, in \S4, we define a Hecke operator acting on VDMFs, preserving the weight of VDMFs, such that its action to the last entry of a VDMF induces a Hecke operator acting on the space of Drinfeld modular forms upon the substitution $t=\theta^{q^n}$ (see \S3.5). Finally, using the effect of our Hecke operator on $u$-expansions, especially its effect on the linear term, in Theorem \ref{T:Hecke}, we show that $\mathcal{G}_r$ is a Hecke eigenform and hence its last entry, $\mathcal{H}_r$, may also be described as a Hecke eigenform upon the substitution $t=\theta^{q^n}$.

\begin{remark}\label{R:intr} Let $u_2$ be a uniformizer at infinity for Drinfeld modular forms of rank two (for the congruence subgroup $\GL_2(A)$) and let $A_{+}$ be the set of monic polynomials in $A$. In the rank two case, Petrov \cite{Pet13} studied a particular class of Drinfeld modular forms having \textit{an $A$-expansion of exponent $n$} which can be uniquely written as 
\[
f(z)=c_0(f)+\sum_{a\in A_{+}}c_a(f)G_n(u_2(az))
\]
where $z\in \Omega^2=\mathbb{C}_{\infty}\setminus K_{\infty}$ is sufficiently large, $c_0(f),c_a(f)\in \mathbb{C}_{\infty}$ and $G_n$ is the $n$-th Goss polynomial (see \cite[\S3]{Gek88} for a thorough discussion on Goss polynomials). He further showed in \cite[Thm. 2.3]{Pet13} that the eigenvalues of Hecke eigenforms having an $A$-expansion can be explicitly described. Later on, Pellarin and Perkins \cite[\S5.2.1]{PP18} recovered a sub-class of Petrov's family of cuspidal Hecke eigenforms with $A$-expansion via specializing a certain entry of VDMFs. We also refer the reader to \cite[\S7.6]{Pel19} for the connection between vectorial Drinfeld modular forms and Petrov's $A$-expansions in the rank two setting. 

In the arbitrary rank setting, the only known examples regarding this direction are the Eisenstein series which are Hecke eigenforms and have a $u$-expansion written in terms of Goss polynomials (see \cite[\S14]{BBP18} and \cite[\S3]{BB17} for details). However, they fail to be a Drinfeld cusp form (see \cite[(3.6)]{BB17}). The main motivation for authors to initiate the present work is to provide cuspidal Hecke eigenforms of arbitrary rank which have a certain form that could be used to generalize the notion of $A$-expansion. In what follows, we make some observation regarding such a generalization. First, we remark that for $j\geq 0$ and $1\leq i \leq r$,  $\mathcal{E}_{q^{j}}^{[i]}(\zz,\cdot)$ may be extended to an entire function of the variable $t$. 
Indeed, using \eqref{E:det1}, we write 
\[
\mathcal{F}(\zz,t)^{-1}=\frac{\tilde{\pi}^{\frac{q^r-1}{q-1}}h_r(\zz)}{\omega(t)}\Cof(\zz,t)^{\tr}
\]
where $\Cof(\zz,t)\in \GL_r(\mathbb{T})$ is the cofactor matrix of $\mathcal{F}(\zz,t)$. By \cite[Def. 3.16, (4.120)]{NP21}, there exists $V\in \GL_r(\mathbb{C}_{\infty})$ such that the entries of $V^{-1}(\mathcal{F}(\zz,t)^{-1})^{(1)}$ are entire functions of $t$. This implies that the entries of $(\omega(t)^{-1}\Cof(\zz,t))^{(1)}$ are entire functions of $t$. Since twisting has no affect on entireness of a function, we then see that the entries of $\omega(t)^{-1}\Cof(\zz,t)$ are also entire functions of $t$. 

Letting $\mathcal{C}_{i1}(\zz,t)$ be the $(i,1)$-th entry of $\Cof(\zz,t)$ and using \eqref{E:det1} and \eqref{E:mainid}, we have 
\[
\mathcal{E}_{1}^{[i]}(\zz,t)=\frac{\tilde{\pi}^{\frac{q^r-1}{q-1}} h_r(\zz)}{(t-\theta)}\omega(t)^{-1}\mathcal{C}_{i1}(\zz,t).
\]
Using Proposition \ref{P:and}(i), one can easily see that $ \mathcal{C}_{i1}(\zz,t) $ has no pole at $t=\theta$. On the other hand, by the definition of $\omega(t)$ (see \eqref{D:omega}), $\omega(t)^{-1}$ has simple zeros at $t=\theta^{q^{\ell}}$ for $\ell\geq 0$ and no other zeros. Thus, the entire function $\omega(t)^{-1}\mathcal{C}_{i1}(\zz,t) $ has a simple zero at $t=\theta$. Hence, $\mathcal{E}_{1}^{[i]}(\zz,t)$ has a removable singularity at $t=\theta$. Therefore,  $\mathcal{E}_{1}^{[i]}(\zz,\cdot)$ may be extended to an entire function of $t$. Using the identity $\mathcal{E}_{q^{j}}^{[i]}(\zz,t)=(\mathcal{E}_{1}^{[i]}(\zz,t))^{(j)}$, we see that $\mathcal{E}_{q^{j}}^{[i]}(\zz,\cdot)$ may be also extended to an entire function of $t$. 

Now, for sufficiently large $\zz\in \Omega^{r}$, by \cite[Cor. 5.2]{CG21}, we have 
\[
\mathcal{E}_{q^{j}}^{[1]}(\zz,t)=(\mathcal{E}_{1}^{[1]}(\zz,t))^{(j)}=-\tilde{\pi}^{q^j}\sum_{a\in A_{+}}a(t)u_a(\zz)^{q^{j}}=-\tilde{\pi}^{q^j}\sum_{a\in A_{+}}a(t)G_1(u_a(\zz))^{q^{j}}.
\]
Here $u_a$ is defined as in \eqref{E:defua} and $G_1(X)=X$ is the Goss polynomial of weight $1$. For $n\geq r-1$ and $0\leq j \leq r-2$, we then have 
\begin{equation}\label{E:eisev}
\mathcal{E}_{q^{j}}^{[1]}(\zz,\theta^{q^n})=-\tilde{\pi}^{q^j}\sum_{a\in A_{+}}a^{q^{n}}u_a(\zz)^{q^{j}}=\left(-\tilde{\pi}\sum_{a\in A_{+}}a^{q^{n-j}}G_1(u_a(\zz))\right)^{q^{j}}.
\end{equation}
 This implies that the entries of the first column of the matrix used to define $\mathcal{H}_r$ are particular $q$-th powers of a certain \textit{$A$-expansion}. Regarding our findings in this paper, we expect that a higher rank analogue of Petrov's result on Drinfeld eigenforms with $A$-expansion may be given in terms of a determinant form as in the definition of $\mathcal{H}_r$. However, there are still details to be understood to strengthen such a generalization. For example, for $2\leq i \leq r$, entries of $i$-th column of the matrix used to define $\mathcal{H}_r$ do not seem to be formed by using $A$-expansions and even worse, by \eqref{E:mainid}, \cite[Cor. 5.9]{CG21} and \cite[Lem. 4.14(i)]{CG22}, they do not have a $u$-expansion. On the other hand, there seems to be a cancellation so that once one takes the determinant, by Theorem \ref{T:main}(i), it provides a Drinfeld modular form which certainly has a $u$-expansion. Moreover, in \cite{Pel19}, Pellarin also developed the notion of expansion at the cusp infinity for vectorial Drinfeld modular forms (see \cite[Chap. 3]{Pel19}) which is also a missing point in the arbitrary rank setting and seems to be crucial to investigate further the $A$-expansions for the higher rank case. We hope to come back to these problems in near future.
\end{remark}
\subsection*{Acknowledgments} The authors would like to thank the anonymous referee for careful reading of the manuscript and valuable suggestions that improve the presentation of the results of the paper. The authors are also grateful to Gebhard B\"{o}ckle, Yen-Tsung Chen, Federico Pellarin and Sjoerd de Vries for their valuable comments. The first author acknowledges support from NSTC Grant 113-2115-M-007-001-MY3. The second author is supported by Academia Sinica Investigator Grant AS-IA-112-M01 and NSTC grant 113-2115-M-001-001. 

\section{Preliminaries}
In this section, our goal is to introduce necessary background on Drinfeld modules and Drinfeld modular forms as well as Hecke operators to obtain our results. Our exposition is mainly based on \cite{Bas23,BBP18,Goss,Pap23, Pel08}.
\subsection{Drinfeld modules and Anderson generating functions} We let $\mathbb{C}_{\infty}\dbl \tau\dbr$ be the ring of power series in $\tau$ with coefficients in $\mathbb{C}_{\infty}$ subject to the condition $\tau \alpha=\alpha^q\tau$
 for all $\alpha\in \mathbb{C}_{\infty}$. Moreover, we denote by $\mathbb{C}_{\infty}[\tau]\subset \mathbb{C}_{\infty}\dbl\tau\dbr$ the ring of polynomials in $\tau$.

\begin{definition} \textit{A Drinfeld module of rank $r\geq 1$} is an $\mathbb{F}_q$-algebra homomorphism $\phi:A\to \mathbb{C}_{\infty}[\tau]$ uniquely described by 
\[
\phi_{\theta}:=\phi(\theta)=\theta+\phi_{\theta,1}\tau+\cdots+\phi_{\theta,r}\tau^r, \ \ \phi_{\theta,r}\neq 0.
\]
For each $1\leq i \leq r$, we call $\phi_{\theta,i}$ \textit{the $i$-th coefficient of $\phi$}.
\end{definition}

For each Drinfeld module $\phi$, there exists a unique infinite series $\exp_{\phi}=\sum_{n\geq 0}\alpha_n\tau^n\in \mathbb{C}_{\infty}\dbl\tau\dbr$ so that $\alpha_0:=1$ and it satisfies the functional equation 
\begin{equation}\label{E:funceqtn}
\exp_{\phi}\theta=\phi_{\theta}\exp_{\phi}.
\end{equation}
The infinite series $\exp_{\phi}$ induces an $\mathbb{F}_q$-linear everywhere convergent function $\exp_{\phi}:\mathbb{C}_{\infty}\to \mathbb{C}_{\infty}$ given by 
\[
\exp_{\phi}(z)=\sum_{n=0}^{\infty}\alpha_nz^{q^n}, \ \ z\in \mathbb{C}_{\infty}.
\]
We further call each non-zero element of $\Ker(\exp_{\phi})$ \textit{a period of $\phi$}.

 One can also construct the formal inverse of $\exp_{\phi}$ in $\mathbb{C}_{\infty}\dbl\tau\dbr$ which we denote by $\log_{\phi}=\sum_{n\geq 0}\beta_n\tau^n\in \mathbb{C}_{\infty}\dbl\tau\dbr$. Similar to $\exp_{\phi}$, the infinite series  $\log_{\phi}$ induces an $\mathbb{F}_q$-linear function $\log_{\phi}:\mathcal{D}_{\phi}\to \mathbb{C}_{\infty}$ given by 
\[
\log_{\phi}(z)=\sum_{n=0}^{\infty}\beta_nz^{q^n}
\]
for all $z$ in a subdomain $\mathcal{D}_{\phi}$ of $\mathbb{C}_{\infty}$.

Let $\Lambda\subset \mathbb{C}_{\infty}$ be a free $A$-module ($\mathbb{F}_q$-vector space respectively). We call $\Lambda$ \textit{an $A$-lattice} (\textit{an $\mathbb{F}_q$-lattice respectively}) if its intersection with any ball of finite radius is finite. Moreover, for any $\alpha\in \mathbb{C}_{\infty}^{\times}$ and $A$-lattice ($\mathbb{F}_q$-lattice respectively), we denote by $c\Lambda$ the $A$-lattice ($\mathbb{F}_q$-lattice respectively) consisting of elements $\alpha \lambda$ for $\lambda \in \Lambda$.

We further define \textit{the exponential function $\exp_{\Lambda}:\mathbb{C}_{\infty}\to \mathbb{C}_{\infty}$ of $\Lambda$}  by 
\[
\exp_{\Lambda}(z):=z\prod_{\lambda\in \Lambda\setminus \{0\}}\left(1-\frac{z}{\lambda}\right), \ \ z\in \mathbb{C}_{\infty}.
\]
Observe that for any $\alpha\in \mathbb{C}_{\infty}^{\times}$, we have 
\begin{equation}\label{E:constantexp}
    \exp_{\alpha\Lambda}(\alpha z)=\alpha\exp_{\Lambda}(z).
\end{equation}

Due to Drinfeld \cite{Dri74}, there exists a one-to-one correspondence between Drinfeld modules of rank $r$ and $A$-lattices of rank $r$. More precisely, each Drinfeld module $\phi$ corresponds to an $A$-lattice $\Lambda_{\phi}$ of rank $r$ so that $\Lambda_{\phi}=\Ker(\exp_{\phi})$. Conversely, each $A$-lattice $\Lambda$ of rank $r$ corresponds to a Drinfeld module $\phi$ of rank $r$ so that $\exp_{\phi}=\exp_{\Lambda}$.

\begin{example}\label{Ex:C} One of the most fundamental examples of Drinfeld modules is given by \textit{the Carlitz module} $C$ defined as 
\[
C_{\theta}:=\theta+\tau.
\]    
Recall that $D_0=1$ and for $\ell\geq 1$, $D_{\ell}=(\theta^{q^{\ell}}-\theta)D_{\ell-1}\in A$. The exponential function $\exp_{C}$ of $C$ is explicitly given by 
\[
\exp_{C}(z)=\sum_{n=0}^{\infty}\frac{z^{q^n}}{D_n},\ \ z\in \mathbb{C}_{\infty}.
\]
The $A$-lattice of rank one corresponding to the Carlitz module is generated over $A$ by the Carlitz period $\tilde{\pi}$ defined as in \eqref{E:pi}. 
\end{example}

Let $z\in \mathbb{C}_{\infty}$. We define \textit{the Anderson generating function} $s_{\phi}(z,t)$ by the infinite series 
\[
s_{\phi}(z,t):=\sum_{n=0}^{\infty}\exp_{\phi}\left(\frac{z}{\theta^{n+1}}\right)t^n\in \mathbb{T}.
\]
We continue with stating fundamental properties of Anderson generating functions which are due to Pellarin (see also \cite[Prop. 3.2]{EGP14}).
\begin{proposition}\cite[\S4.2]{Pel08}, \label{P:and} Let $\exp_{\phi}=\sum_{n\geq 0}\alpha_n\tau^n\in \mathbb{C}_{\infty}\dbl\tau\dbr$ be the exponential series of $\phi$. The following statements hold.
\begin{itemize}
    \item[(i)] We have an identity in $\mathbb{T}$:
    \[
    s_{\phi}(z,t)=\sum_{n=0}^{\infty}\frac{\alpha_n z^{q^n}}{\theta^{q^n}-t}.
    \]
    \item[(ii)] As a function of $t$, for each $n\geq 0$, $s_{\phi}(z,t)$ has a simple pole at $t=\theta^{q^n}$ with residues  $\Res_{t=\theta^{q^n}}s_{\phi}(z,t)=-\alpha_nz^{q^n}$.
    \item[(iii)] For any $z\in \Ker(\exp_{\phi})$, we have 
    \[
    (t-\theta)s_{\phi}(z,t)=\phi_{\theta,1}s_{\phi}^{(1)}(z,t)+\cdots+\phi_{\theta,r}s_{\phi}^{(r)}(z,t).
    \]
\end{itemize}
    
\end{proposition}
As a fundamental example of Anderson generating functions, we define \textit{the Anderson-Thakur function}, which is the Anderson generating function $s_{C}(\tilde{\pi},t)$, by the infinite product
\begin{equation}\label{D:omega}
\omega(t):=s_{C}(\tilde{\pi},t)=(-\theta)^{1/(q-1)}\prod_{i=0}^{\infty}\left(1-\frac{t}{\theta^{q^i}}\right)^{-1}\in \mathbb{T}^{\times}.
\end{equation}
As Proposition \ref{P:and}(iii) implies, we have  
\begin{equation}\label{E:omegafunceq}
(t-\theta)\omega(t)=\omega(t)^{(1)}.
\end{equation}
By Proposition \ref{P:and}(ii), we further observe that  
\[
\tilde{\pi}=-\Res_{t=\theta}\omega(t).
\]
\begin{remark} We refer the reader to \cite{KP23} for a study of product expansion of Anderson generating functions as in \eqref{D:omega} corresponding to Drinfeld modules of arbitrary rank.
\end{remark}
\subsection{Drinfeld modular forms} Recall the Drinfeld upper half plane $\Omega^r$ from \S1.  Let $H$ be a $K_\infty$-rational hyperplane in $\mathbb{P}^{r-1}(\mathbb{C}_\infty)$ which is given by a solution set of a linear form
\[
\ell_H(X_1,\dots,X_r):=b_1X_1+\cdots+b_rX_r\in K_\infty[X_1,\dots,X_r]
\]
so that $\max_{1\leq i\leq r}\{\inorm{b_i}\}=1$. For any $\mathbf{x}=(\mathbf{x}_1,\dots,\mathbf{x}_r)^{\tr}\in \mathbb{C}_{\infty}^r$, we set $|\mathbf{x}|_{\infty}:=\max_{i=1}^r\inorm{\mathbf{x}_i}$. Moreover, if we let $\ell_H(\mathbf{x}):=b_1\mathbf{x}_1+\dots+b_r\mathbf{x}_r$, then $\inorm{\ell_{H}(\zz)}$ is well-defined for any $\zz\in\Omega^r$ and is independent of the choice of $\ell_H$. For each $n\geq 1$, we define
\[
\Omega_n^r:=  \{\zz\in\Omega^r\mid\inorm{\ell_H(\zz)}\geq q^{-n}|\zz|_{\infty} \ \ \text{for any $K_{\infty}$-rational hyperplane $H$}\}.
\]
   Note that $\{\Omega_{n}^r\}_{n=1}^{\infty}$ is an admissible covering of $\Omega^r$ (see \cite[\S3]{BBP18} for more details). We call $f:\Omega^r\to \mathbb{C}_{\infty}$ \textit{a rigid analytic function} if its restriction to each $\Omega_n^r$ is the uniform limit of rational functions having no pole in $\Omega_n^r$.

For any $\gamma=(a_{ij})\in \GL_r(K_{\infty})$ and $\zz=(z_1,\dots,z_r)^{\tr}\in \Omega^r$, we define 
\[
j(\gamma;\zz):=a_{r1}z_1+\dots+a_{rr}z_r.
\]
Furthermore, we define the action of $\GL_r(K_{\infty})$ on $\Omega^r$ by 
\[
\gamma\cdot \zz:=\left(
    \frac{a_{11}z_1+\dots+a_{1r}z_r}{j(\gamma;\zz)},
    \dots,\frac{a_{(r-1)1}z_1+\dots+a_{(r-1)r}z_r}{j(\gamma;\zz)},1\right)^{\tr}
\in \Omega^r.
\]
Finally, for any $\gamma\in \GL_r(K_{\infty})$ and a rigid analytic function $f:\Omega^r\to \mathbb{C}_{\infty}$, we define the $k$-th slash operator $|_{k}$ and the $k$-th slash operator $|_{k,m}$ with respect to type $m\in \mathbb{Z}/(q-1)\mathbb{Z}$ by 
\[
(f|_{k}\gamma)(\zz):=j(\gamma;\zz)^{-k}f(\gamma\cdot \zz).
\]
and
\[
(f|_{k,m}\gamma)(\zz):=j(\gamma;\zz)^{-k}\det(\gamma)^{m}f(\gamma\cdot \zz)
\]
respectively.
\begin{definition} Let $k\in \mathbb{Z}$ and $m\in \mathbb{Z}/(q-1)\mathbb{Z}$. We call a rigid analytic function $f:\Omega^r\to \mathbb{C}_{\infty}$ \textit{a weak Drinfeld modular form of weight $k$ and type $m$ (for $\GL_r(A)$)} if it satisfies
\[
f|_{k,m}\gamma=f
\]
for all $\gamma \in \GL_r(A)$. We denote by $\mathcal{WM}_{k}^m$ the $\mathbb{C}_{\infty}$-vector space spanned by the weak Drinfeld modular forms of weight $k$ and type $m$. We  also let $\mathcal{WM}_k$ be the $\mathbb{C}_{\infty}$-vector space spanned by all the weak Drinfeld modular forms of weight $k$ and any type.
\end{definition}

For  $(a_1,\dots,a_{r-1})\in A^{r-1}$, we set 
\begin{equation}\label{E:ainv}
\gamma_{a_1,\dots,a_{r-1}}:=\begin{pmatrix}
    1&a_1& & a_{r-1}\\
     &\ddots & & \\
     & & \ddots & \\
     & & & 1
\end{pmatrix}\in \GL_r(A).
\end{equation}
A rigid analytic function $f$ is called \textit{$A$-invariant} if 
$
f(\gamma_{a_1,\dots,a_{r-1}}\cdot \zz)=f(\zz)
$
for any tuple $(a_1,\dots,a_{r-1})\in A^{r-1}$. For any $\zz=(z_1,z_2,\dots,z_r)^{\tr}\in \Omega^r$, we let $\tz:=(z_2,\dots,z_r)^{\tr}\in \Omega^{r-1}$ and, by a slight abuse of notation, write $\zz=(z_1,\tz)\in \Omega^r$. Consider the $A$-lattice $\Lambda_{\tz}$ of rank $r-1$ generated by the elements $z_2,\dots,z_r$. We further define \textit{the uniformizer at infinity} by
\[
u(\zz):=\frac{1}{\exp_{\tilde{\pi}\Lambda_{\tz}}(\tilde{\pi}z_1)}=\frac{1}{\tilde{\pi}}\frac{1}{\exp_{\Lambda_{\tz}}(z_1)}
\]
where the last equality follows from \eqref{E:constantexp}. 

Let $f$ be an $A$-invariant function. By \cite[Prop. 5.4]{BBP18}, for each $n\in \mathbb{Z}$, we know that there exists a unique rigid analytic function $f_n:\Omega^{r-1}\to\mathbb{C}_{\infty}$ such that 
\begin{equation}\label{E:uexp}
f(\zz)=\sum_{n\in \mathbb{Z}}f_n(\tz)u(\zz)^n
\end{equation}
whenever $\zz$ lies in some neighborhood of infinity. Moreover, by \cite[Prop. 3.2.]{Bas14}, each $f_n$ is a weak Drinfeld modular form of weight $k-n$ (for $\GL_{r-1}(A))$. Since $\Omega^r$ is a connected rigid analytic space, the right hand side of \eqref{E:uexp} uniquely determines $f$. We call such an expansion \textit{the $u$-expansion of $f$}. Moreover, an $A$-invariant function $f$ is called \textit{holomorphic at infinity} if $f_n$ is identically zero whenever $n<0$ and  \textit{meromorphic at infinity} if $f_n$ is identically zero for all but finitely many $n<0$.

\begin{definition} Let $k\in \mathbb{Z}_{\geq 0}$ and $m\in \mathbb{Z}/(q-1)\mathbb{Z}$. We call $f\in \mathcal{WM}_{k}^{m}$ \textit{a Drinfeld modular form of weight $k$ and type $m$ (for $\GL_r(A)$)} if $f$ is holomorphic at infinity. Furthermore, we call $f$ \textit{a Drinfeld cusp form} if the function $f_0$ that appears in its $u$-expansion is identically zero. We denote by $\mathcal{M}_{k}^m$ the $\mathbb{C}_{\infty}$-vector space spanned by Drinfeld modular forms of weight $k$ and type $m$. We also let $\mathcal{M}_k$ be the $\mathbb{C}_{\infty}$-vector space spanned by all the Drinfeld modular forms of weight $k$ and any type.
\end{definition}

In what follows, we provide several examples of Drinfeld modular forms.
\begin{example}\label{Ex:1}
	\begin{itemize}
		\item[(i)] For each $\zz=(z_1,\dots,z_r)^{\tr}\in \Omega^r$, let $A^r\zz$ be the $A$-lattice of rank $r$ generated by the entries of $\zz$ and let $\phi^{\zz} $ be the Drinfeld module corresponding to $A^r\zz$. We write 
		\[\phi^{\zz}_\theta:=\theta+g_1(\zz)\tau+\dots+g_{r}(\zz)\tau^{r}.
		\]
		Then, for any $1\leq \mu \leq r$, the function $g_{\mu}:\Omega^r\to \CC_{\infty}$, which we call \textit{the $\mu$-th coefficient form}, sending $\zz\mapsto g_\mu(\zz)$ is a Drinfeld modular form of weight $q^{\mu}-1$ and type $0$ \cite[Prop. 15.12]{BBP18}. We further call $g_r$ \textit{the discriminant function of rank $r$}  which is indeed a nowhere-vanishing Drinfeld cusp form on $\Omega^r$. 
		\item[(ii)] Let $\exp_{\phi^{\zz}}:=\sum_{n\geq 0}\alpha_n(\zz)\tau^n$ be the exponential series of $\phi^{\zz}$. For each $n\geq 1$, the function $\alpha_n:\Omega^r\to \mathbb{C}_{\infty}$ sending each $\zz\in \Omega^r$ to $\alpha_n(\zz)$ is a Drinfeld modular form of weight $q^n-1$ and type $0$ \cite[Prop. 15.3]{BBP18}. Similarly, letting $\log_{\phi^{\zz}}:=\sum_{n\geq 0}\beta_n(\zz)\tau^n$ be the logarithm series of $\phi^{\zz}$, for each $n\geq 1$, the function $\beta_n:\Omega^r\to \mathbb{C}_{\infty}$ sending each $\zz\in \Omega^r$ to $\beta_n(\zz)$ is an element in $\mathcal{M}_{q^n-1}^{0}$.
        
		\item[(iii)] We now introduce a non-zero type Drinfeld modular form which we call \textit{Gekeler's $h$-function} (see \cite{Gek17} for more details). By convention, we set $h_1:=-1$ and let $r\geq 2$. Consider the tuple $\nu=(\nu_1,\dots,\nu_r) \in \mathbb{F}_q^{r}\setminus \{(0,\dots,0)\}$. We call $\nu$ \textit{monic} if $\nu_{\ell}=1$ for the largest subscript $\ell$ satisfying $\nu_\ell\neq 0$. For any  $\zz=(z_1,\dots,z_r)^{\tr}\in \Omega^r$, we define 
		\[
		\Eis_{\nu}(\zz):=\sum_{\substack{(a_1,\dots,a_r)\in K^{r}\\ (a_1,\dots,a_r)\equiv \theta^{-1}\nu \pmod{A^r} }}\frac{1}{a_1z_1+\dots+a_rz_r}.
		\]
		 Then \textit{Gekeler's $h$-function} $h_r:\Omega^r\to \CC_{\infty}$  is given by
		\[
	h_r(\zz):=\tilde{\pi}^{\frac{1-q^r}{q-1}}(-\theta)^{1/(q-1)}\prod_{\substack{\nu\in \mathbb{F}_q^{r}\\\nu \text{ monic}}}\Eis_{\nu}(\zz).
	\] 
We note that it is a nowhere-vanishing Drinfeld cusp form  of weight $(q^r-1)/(q-1)$ and type $1$. Furthermore, by \cite[Cor. 6.3]{BB17}, whenever $\zz$ lies in some neighborhood of infinity, we have 
\begin{equation}\label{E:huexp}
h_r(\zz)=(-1)^rh_{r-1}(\tz)^qu+O(u^q)
\end{equation}
where we mean by $O(u^q)$ a particular element of the ring of power series in $u$ with coefficients in rigid analytic functions on $\Omega^{r-1}$ so that the coefficient of each $u^i$ with $i<q$ is zero. 
\end{itemize}
\end{example}

\subsection{Hecke operators on Drinfeld modular forms} From now on, we will fix a non-constant monic irreducible polynomial $\mathfrak{p}\in A$. Consider 
\[
\delta:=\begin{pmatrix}
    \mathfrak{p}&\\
     & \Id_{r-1}
\end{pmatrix}\in \GL_r(K)\cap \Mat_r(A)
\]
which is a diagonal matrix whose first entry is $\mathfrak{p}$ and the other diagonal entries are equal to $1$.

For each $\ell\in \{1,\dots,r\}$ and a tuple $\bb=(b_1,\dots,b_r)\in A^r$, we define the matrix $\beta_{\ell,\bb}\in \GL_r(K)\cap \Mat_{r}(A)$ given by
\[
\beta_{\ell,\bb}:=\begin{pmatrix}
    1& &  & b_1& & \\
     & 1 & & b_2 & & \\
     & & \ddots & \vdots & & \\
      & &  & \vdots & \ddots& \\
      & & & b_r & & 1
\end{pmatrix}
\]
so that the $\ell$-th column is $(b_1,\dots,b_r)^{\tr}$ and for $i\neq \ell$, the $i$-th column is the $i$-th unit vector. Furthermore, we let $\widetilde{\beta_{\ell,\bb}}\in \GL_{r-1}(K)\cap \Mat_{r-1}(A)$ be the matrix formed by removing the first row and the first column of $\beta_{\ell,\bb}$.  Moreover, for any $\bb=(b_1, \dots, b_r) \in A^r$, we define $\mathbf{b}_{\ell}:=(b_1,\dots,b_{\ell-1},\mathfrak{p},0,\dots,0)\in A^{r}$.
Then we set 
\[
B_{\ell,r}:=\{\beta_{\ell,\bb_{\ell}} \ \ | \ \ \bb=(b_1,\dots,b_r)\in A^r, \ \ \deg_{\theta}(b_i)<\deg_{\theta}(\mathfrak{p}) \text{ for all $i<\ell$}\}.
\]

The following result of Basson is crucial to determine our Hecke operators explicitly.
\begin{proposition}\cite[Prop. 3.3]{Bas23}\label{P:repr} The union $\cup_{\ell=1}^rB_{\ell,r}$ is a set of coset representatives for $\GL_r(A)\setminus \GL_r(A)\delta \GL_r(A)$.
\end{proposition}
Now we are ready to construct our Hecke operators.
\begin{definition}\label{D:Hecke} We define our Hecke operator $T_{\mathfrak{p},r}:\mathcal{WM}_k\to \mathcal{WM}_k$ by 
\[
T_{\mathfrak{p},r}(f):=\mathfrak{p}^k\sum_{\gamma\in \cup_{\ell=1}^rB_{\ell,r}}f|_{k}\gamma , \ \ f\in \mathcal{WM}_k.
\]
 \end{definition}
By \cite[Cor. 3.12]{Bas23}, we know that $ T_{\mathfrak{p},r} $ induces an endomorphism of $\mathcal{M}_k$ for each $k\geq 0$.

\begin{remark} \label{R:Hecke}
\begin{itemize}
\item[(i)] We emphasize that although, in this subsection, our exposition is mainly based on \cite[\S3]{Bas23}, we use $\mathfrak{p}^k$-multiple of the operator studied by Basson. Since $T_{\mathfrak{p},r}$ is $\mathbb{C}_{\infty}$-linear, this modification does not effect the proofs of the results obtained in \cite[\S3]{Bas23}. Our main reason to use the operator defined  in Definition \ref{D:Hecke} instead is to be compatible with the Hecke operators studied by Gekeler in \cite[\S7]{Gek88} and Pellarin and Perkins in \cite[\S5]{PP18}.
\item[(ii)] We note that our notation $T_{\mathfrak{p},r}$ does not include $k$. However, since, throughout the paper, we will analyze $T_{\mathfrak{p},r}(f)$ for some $f\in \mathcal{WM}_k$, the integer $k$ should be clear from the weight of the weak Drinfeld modular form $f$. 
\end{itemize}
\end{remark}

To describe the effect of $T_{\mathfrak{p},r}$ on the $u$-expansion of a Drinfeld modular form, first we need to introduce some more notation. 
For $2\leq \ell \leq r$, we let
\[
\Lambda_{\ell,\bb_{\ell}}^{'}:=A^{r-1}\widetilde{\beta_{\ell,\bb_{\ell}}}\cdot \tz, 
\]
which is the $A$-lattice of rank $r-1$ generated by the last $(r-1)$-th entries of $ \beta_{\ell,\bb_{\ell}}
\cdot \zz $. More precisely, observe that, for $2\leq \ell \leq r-1$, the $A$-lattice $\Lambda'_{\ell,\bb_{\ell}}$ is generated by $z_2+b_2z_{\ell},\dots,z_{\ell-1}+b_{\ell-1}z_{\ell},\mathfrak{p}z_{\ell}$ $,z_{\ell+1},\dots,z_r$ over $A$ and therefore, $\Lambda'_{\ell,\bb_{\ell}}$ is a sub $A$-lattice of $\Lambda_{\tz}$ so that 
\begin{equation}\label{E:lattice1}
\Lambda_{\tz}/\Lambda'_{\ell,\bb_{\ell}}\cong A/\mathfrak{p}A.
\end{equation}
When $\ell=r$, the $A$-lattice $\Lambda'_{r,\bb_{r}}$ is generated by $\frac{z_2+b_2}{\mathfrak{p}},\dots,\frac{z_{r-1}+b_{r-1}}{\mathfrak{p}}$ and $1$ over $A$ and hence $\Lambda_{\tz}$ is a sub $A$-lattice of $\Lambda'_{r,\bb_{r}}$ so that 
\[
\Lambda'_{r,\bb_{r}}/\Lambda_{\tz}\cong (A/\mathfrak{p}A)^{r-2}.
\]
Note that  $\mathfrak{p}\Lambda'_{r,\bb_{r}}$ is is generated by $z_2+b_2,\dots,z_{r-1}+b_{r-1}$ and $\mathfrak{p}$  and hence is a sub $A$-lattice of $\Lambda_{\tz}$ so that 
\begin{equation}\label{E:lattice2}
\Lambda_{\tz}/\mathfrak{p}\Lambda'_{r,\bb_{r}}\cong A/\mathfrak{p}A.
\end{equation}
 Moreover, we  define
\[
L_{\ell,\bb_{\ell}}:=\begin{cases}
    \exp_{\tilde{\pi}\Lambda_{\ell,\bb_{\ell}}^{'}}(\tilde{\pi}\Lambda_{\tz}) & \text{ if } \ell=2,\dots,r-1\\
 \exp_{\mathfrak{p}\tilde{\pi}\Lambda_{\ell,\bb_{\ell}}^{'}}(\tilde{\pi}\Lambda_{\tz}) & \text{ if } \ell=r.    
\end{cases}
\]
Thus, by \eqref{E:lattice1}, \eqref{E:lattice2} and \cite[Prop. 2.3(a)]{BBP18}, each $L_{\ell,\mathbf{b}_{\ell}}$ is isomorphic to $ A/\mathfrak{p}A$ and hence forms a finite dimensional $\mathbb{F}_q$-vector space. Finally, for any $a\in A\setminus\{0\}$, we let 
\begin{equation}\label{E:defua}
u_{a}(\zz):=\exp_{\tilde{\pi}\Lambda_{\tz}}(a\tilde{\pi}z_1)^{-1}.
\end{equation}

In what follows, we will investigate the affect of Hecke operators on the $u$-expansion of Drinfeld modular forms. Before we state the next theorem, for an $\mathbb{F}_q$-lattice $\Lambda$ and $n\geq 0$, we let $G_{n,\Lambda}(X)\in \mathbb{C}_{\infty}[X]$ be the $n$-th Goss polynomial corresponding to $\Lambda$. For more details on Goss polynomials, we refer the reader to \cite[\S3]{Gek88}.

\begin{theorem}[{cf. \cite[Prop. 4.9]{Gek26}}]\cite[Thm. 3.6, Cor. 3.11]{Bas23}\label{T:HeckeDr} Let $f\in \mathcal{M}_k$ so that its $u$-expansion is given by $f=\sum_{n=0}^{\infty}f_nu^n$. Then, for any $\zz\in \Omega^{r}$ lying in some neighborhood of infinity, we have 
\begin{equation}\label{E:heckeu}
T_{\mathfrak{p},r}f(\zz)=\mathfrak{p}^k\sum_{n=0}^{\infty}f_n(\tz)u_{\mathfrak{p}}(\zz)^n+\mathfrak{p}^k\sum_{n=0}^{\infty}\sum_{\substack{2\leq \ell \leq r\\ \beta_{\ell, \bb_{\ell}}\in B_{\ell,r}}}(f_n|_{k-n}\widetilde{\beta_{\ell, \bb_{\ell}}})( \tz)G_{n,L_{\ell,\bb_{\ell}}}(u(\zz)).
\end{equation}
In particular, the constant term of the $u$-expansion of $T_{\mathfrak{p},r}(f)$ is $T_{\mathfrak{p},r-1}(f_0)$ and the linear term of the $u$-expansion of $T_{\mathfrak{p},r}(f)$ is equal to $\mathfrak{p}T_{\mathfrak{p},r-1}(f_1)$.

\end{theorem}
\begin{proof} Since the idea of the proof of the theorem will be later used in \S4, we provide its crucial details for completeness. We first let $\ell=1$. Since $\bb_1=(\mathfrak{p},0,\dots,0)$ and $j(\beta_{1,\bb_1};\zz)=1$, we see that 
\[
(f|_{k}\beta_{1,\bb_1})( \zz)=\sum_{n=0}^{\infty}f_n(\tz)u_{\mathfrak{p}}(\zz)^n,
\]
that is, the first summation in the right hand side of \eqref{E:heckeu} corresponds to $\ell=1$. 

Now, let $2\leq \ell <r$. Note that 
\[
\beta_{\ell,\bb_{\ell}}\cdot \zz=(z_1+b_1z_{\ell},\dots, z_{\ell-1}+b_{\ell-1}z_{\ell},\mathfrak{p}z_{\ell},z_{\ell+1},\dots,z_r)^{\tr}\in \Omega^{r}.
\]
Since $j(\beta_{\ell, \bb_{\ell}};\zz)=1$, we also have
\[
(f|_{k}\beta_{\ell,\bb_{\ell}})(\zz)=f(\beta_{\ell,\bb_{\ell}}\cdot \zz).
\]
For any $\bb_{\ell}=(b_1, \dots, b_{\ell-1}, \mathfrak{p}, 0, \dots, 0)\in A^r$, we define  $\widetilde{\bb}_{\ell}:=(b_2,\dots,b_{\ell-1},\mathfrak{p},0,\dots,0)\in A^{r-1}$ so that $\bb_{\ell}=(b_1, \widetilde{\bb}_{\ell})$. For each $n\geq 0$, consider the summation given by
\[
\sum_{\substack{b_1\in A\\\deg_{\theta}(b_1)<\deg_{\theta}(\mathfrak{p})}}f_n(\widetilde{\beta_{\ell,(b_1,\widetilde{\bb}_{\ell})}}\cdot \tz)(\exp_{\tilde{\pi}\Lambda'_{\ell,(b_1,\widetilde{\bb}_{\ell})}}(\tilde{\pi}(z_1+b_1z_{\ell})))^{-n}.
\]
Here, we remind the reader that 
\[
\widetilde{\beta_{\ell,\bb_{\ell}}}\cdot \tz=(z_2+b_2z_{\ell},\dots,z_{\ell-1}+b_{\ell-1}z_{\ell},\mathfrak{p}z_{\ell},z_{\ell+1},\dots,z_{r})^{\tr}\in \Omega^{r-1}
\]
and hence the term $f_n(\widetilde{\beta_{\ell,(b_1,\widetilde{\bb}_{\ell})}}\cdot \tz)$ has no dependence on $b_1$ and $z_1$. To calculate the resulting $u$-expansion, we concentrate on the term $ \exp_{\tilde{\pi}\Lambda'_{\ell,(b_1, \widetilde{\bb}_{\ell})}}(\tilde{\pi}(z_1+b_1z_{\ell})) $. First, we observe that 
\begin{multline}\label{E:gosspolynomialcalc}
\sum_{\substack{b_1\in A\\\deg_{\theta}(b_1)<\deg_{\theta}(\mathfrak{p})}}\exp_{\tilde{\pi}\Lambda'_{\ell,(b_1,\widetilde{\bb}_{\ell})}}(\tilde{\pi}(z_1+b_1z_{\ell}))^{-n}\\
=\sum_{\substack{b_1\in A\\\deg_{\theta}(b_1)<\deg_{\theta}(\mathfrak{p})}}\frac{1}{(\exp_{\tilde{\pi}\Lambda'_{\ell,(b_1,\widetilde{\bb}_{\ell})}}(\tilde{\pi}z_1)+\exp_{\tilde{\pi}\Lambda'_{\ell,(b_1,\widetilde{\bb}_{\ell})}}(\tilde{\pi}b_1z_{\ell}))^n}
\\=
\sum_{v\in L_{\ell,(b_1,\widetilde{\bb}_{\ell})}}\frac{1}{(\exp_{\tilde{\pi}\Lambda'_{\ell,(b_1,\widetilde{\bb}_{\ell})}}(\tilde{\pi}z_1)+v)^n}.
\end{multline}
Here the last equality follows from the fact that all the elements of $L_{\ell,(b_1,\widetilde{\bb}_{\ell})}$ are of the form $\exp_{\tilde{\pi}\Lambda'_{\ell,(b_1,\widetilde{\bb}_{\ell})}}(\tilde{\pi}b_1z_{\ell}) $ where $b_1\in A$ such that $\deg_{\theta}(b_1)<\deg_{\theta}(\mathfrak{p})$.
On the other hand, note that
\begin{multline}\label{E:sum2}
\sum_{v\in L_{\ell,(b_1,\widetilde{\bb}_{\ell})}}\frac{1}{\exp_{\tilde{\pi}\Lambda'_{\ell,(b_1,\widetilde{\bb}_{\ell})}}(\tilde{\pi}z_1)+v}=\sum_{\substack{b_1\in A\\\deg_{\theta}(b_1)<\deg_{\theta}(\mathfrak{p})}}\exp_{\tilde{\pi}\Lambda'_{\ell,(b_1,\widetilde{\bb}_{\ell})}}(\tilde{\pi}(z_1+b_1z_{\ell}))^{-1}\\=\sum_{\substack{b_1\in A\\\deg_{\theta}(b_1)<\deg_{\theta}(\mathfrak{p})}}\sum_{\lambda\in \tilde{\pi}\Lambda'_{\ell,(b_1,\widetilde{\bb}_{\ell})}}\frac{1}{\tilde{\pi}z_1+\tilde{\pi}b_1z_{\ell}+\lambda}=\exp_{\tilde{\pi}\Lambda_{\widetilde\zz}}(\tilde{\pi}z_1)^{-1}=u(\zz).
\end{multline}
Here, the first equality follows from \eqref{E:gosspolynomialcalc}, the second equality from first taking the logarithmic derivative of both sides of
\[
\exp_{\tilde{\pi}\Lambda'_{\ell,(b_1,\widetilde{\bb}_{\ell})}}(X)=X\prod_{\lambda\in \tilde{\pi}\Lambda'_{\ell,(b_1,\widetilde{\bb}_{\ell})}}\left(1-\frac{X}{\lambda}\right)
\]
with respect to $X$ and then letting $X=\tilde{\pi}(z_1+b_1z_{\ell})$. Finally, the third equality follows from the fact that the $A$-lattice $\tilde{\pi}\Lambda_{\tz}$ consists of elements of the form $\tilde{\pi}(b_1z_{\ell}+\lambda)$ where $b_1\in A$ varies through polynomials whose degree less than $\deg_{\theta}(\mathfrak{p})$ and $\lambda$ varies through elements of $ \Lambda'_{\ell,(b_1,\widetilde{\bb}_{\ell})}$. Thus, combining \eqref{E:gosspolynomialcalc} and \eqref{E:sum2} with \cite[Prop. 3.4]{Gek88} (see also \cite[Prop. 2.1]{Bas23}), we have
\[
\sum_{\substack{b_1\in A\\\deg_{\theta}(b_1)<\deg_{\theta}(\mathfrak{p})}}\exp_{\tilde{\pi}\Lambda'_{\ell,(b_1,\widetilde{\bb}_{\ell})}}(\tilde{\pi}(z_1+b_1z_{\ell}))^{-n}
= G_{n,L_{\ell,(b_1,\widetilde{\bb}_{\ell})}}(u(\zz))
\]
and hence we obtain
\begin{equation*}
     \sum_{\substack{b_1\in A\\\deg_{\theta}(b_1)<\deg_{\theta}(\mathfrak{p})}}f_n(\widetilde{\beta_{\ell,(b_1,\widetilde{\bb}_{\ell})}\cdot \zz})\exp_{\tilde{\pi}\Lambda'_{\ell,(b_1,\widetilde{\bb}_{\ell})}}(\tilde{\pi}(z_1+b_1z_{\ell}))^{-n}= f_n(\widetilde{\beta_{\ell,(b_1,\widetilde{\bb}_{\ell})}\cdot \zz})G_{n,L_{\ell,(b_1,\widetilde{\bb}_{\ell})}}(u(\zz)).
\end{equation*}

What remains to analyze is the case $\ell=r$. In this case, $j(\beta_{r,\bb_r};\zz)=\mathfrak{p}$ and 
\[
\beta_{r,\bb_r}\cdot\zz=\left(\frac{z_1+b_1}{\mathfrak{p}},\dots,\frac{z_{r-1}+b_{r-1}}{\mathfrak{p}},1\right)^{\tr}\in \Omega^r.
\]
We now have 
\[
(f|_k\beta_{r,\bb_r})(\zz)=j(\beta_{r,\bb_r};\zz)^{-k}f(\beta_{r,\bb_r}\cdot\zz)=\mathfrak{p}^{-k}f(\beta_{r,\bb_r}\cdot\zz).
\]
Note that, for any non-zero $z\in \mathbb{C}_{\infty}$, by \eqref{E:constantexp}, we have
\begin{equation}\label{E:rel1}
\exp_{\mathfrak{p} \tilde{\pi}\Lambda'_{r,\bb_r}}(\mathfrak{p}\tilde{\pi}z)=\mathfrak{p}\exp_{\tilde{\pi}\Lambda'_{r,\bb_r}}(\tilde{\pi}z).
\end{equation}
We now obtain
\begin{multline*}
\sum_{\substack{b_1\in A\\\deg_{\theta}(b_1)<\deg_{\theta}(\mathfrak{p})}}
    \mathfrak{p}^{-k}f_n(\widetilde{\beta_{r,\bb_r}\cdot \zz})\exp_{\tilde{\pi}\Lambda'_{r,\bb_r}}\left(\frac{\tilde{\pi}(z_1+b_1)}{\mathfrak{p}}\right)^{-n}\\
    =\mathfrak{p}^{-k}f_n(\widetilde{\beta_{r,\bb_r}\cdot \zz})\mathfrak{p}^n\sum_{\substack{b_1\in A\\\deg_{\theta}(b_1)<\deg_{\theta}(\mathfrak{p})}}\exp_{\mathfrak{p} \tilde{\pi}\Lambda'_{r,\bb_r}}(\tilde{\pi}(z_1+b_1))^{-n}\\
    =\mathfrak{p}^{n-k}f_n(\widetilde{\beta_{r,\bb_r}\cdot \zz})G_{n,L_{r,\bb_r}}(u(\zz))
    =(f_n|_{k-n}\widetilde{\beta_{r,\bb_r}})( \tz)G_{n,L_{r,\bb_r}}(u(\zz)).
    \end{multline*}
    Here the first equality follows from \eqref{E:rel1} and the second equality follows from a similar calculation as in \eqref{E:gosspolynomialcalc} and \eqref{E:sum2} as well as \cite[Prop. 3.4]{Gek88} (see also \cite[Prop. 2.1]{Bas23}). Thus, we establish the first assertion of the theorem. To prove the second assertion, we recall that $f_0$ ($f_1$ respectively) is a weak Drinfeld modular form of weight $k$ ($k-1$ respectively). On the other hand, for $n\geq 2$, again by \cite[Prop. 3.4]{Gek88}, the $n$-th Goss polynomial $G_{n,\Lambda}$ associated to any $A$-lattice $\Lambda$ has no constant term and linear term. Therefore, the only contribution to the coefficient of the constant (linear respectively) term in the $u$-expansion of $T_{\mathfrak{p}}(f)$ comes from $T_{\mathfrak{p},r-1}(f_0)$ ($\mathfrak{p}T_{\mathfrak{p},r-1}(f_1)$ respectively) as desired.
\end{proof}

We obtain the following proposition derived from the work of Gekeler \cite[Cor. 7.6]{Gek88} and Theorem \ref{T:HeckeDr}. We remark that later in Corollary \ref{C:main2}, we recover Proposition \ref{P:hfunc} as a special case and hence provide another proof for the same result.

\begin{proposition}[{cf. \cite[Prop. 5.9]{Gek26}}]\label{P:hfunc} We have 
\[
T_{\mathfrak{p},r}(h_r)(\zz)=\mathfrak{p}^{1+q+\cdots+q^{r-2}}h_r(\zz).
\]    
\end{proposition}
\begin{proof} We proceed by induction on $r\geq 2$. When $r=2$, the proposition holds by \cite[Cor. 7.6]{Gek88}. Assume that the proposition holds for $r-1$. On one hand, by \cite[Prop. 16.14(b), Thm. 17.6]{BBP18}, $h_r\in \mathcal{M}_{(q^r-1)/(q-1)}^{1}$ is the Drinfeld cusp form of smallest non-zero weight which generates $ \mathcal{M}_{(q^r-1)/(q-1)}^{1}$ over $\mathbb{C}_{\infty}$. On the other hand, by \cite[Cor. 3.12]{Bas23}, $T_{\mathfrak{p},r}$ is an operator preserving the weight of the Drinfeld modular forms and sends cusp forms to cusp forms. Thus, we see that $T_{\mathfrak{p},r}(h_r)=\alpha h_r$ for some $\alpha\in \mathbb{C}_{\infty}$. Using \eqref{E:huexp}, the induction hypothesis and Theorem \ref{T:HeckeDr}, we see that the linear term of the $u$-expansion of $T_{\mathfrak{p},r}(h_r)$ is 
\begin{multline*}
\alpha (-1)^rh_{r-1}^q=\mathfrak{p}T_{\mathfrak{p},r-1}((-1)^r h_{r-1}^q)=(-1)^r\mathfrak{p}T_{\mathfrak{p},r-1}(h_{r-1}^q)=(-1)^r\mathfrak{p} (T_{\mathfrak{p},r-1}(h_{r-1}))^q\\
=(-1)^r\mathfrak{p}(\mathfrak{p}^{1+q+\dots+q^{r-3}}h_{r-1})^q= \mathfrak{p}^{1+q+\cdots+q^{r-2}} (-1)^rh_{r-1}^q.
\end{multline*}
Here, to avoid the ambiguity, we note that $T_{\mathfrak{p},r-1}$ is used to denote the Hecke operator acting on $\mathcal{M}_{\frac{q^r-q}{q-1}}$ (applied to $(-1)^rh_{r-1}^q$ and $h_{r-1}^q$) and $\mathcal{M}_{\frac{q^{r-1}-1}{q-1}}$ (applied to $h_{r-1}$). Moreover, the third equality follows from the fact that (compare with \cite[Prop. 3.10]{Gek26})

\begin{multline*}
T_{\mathfrak{p},r-1}(h_{r-1}^q)(\zz)=\mathfrak{p}^{\frac{q^r-q}{q-1}}\sum_{\beta_{\ell,\bb}\in \cup_{\ell=1}^rB_{\ell,r}}h_{r-1}(\beta_{\ell,\bb}\cdot \zz)^{q}j(\beta_{\ell,\bb},\zz)^{\frac{q^r-q}{q-1}}\\=\left(\mathfrak{p}^{\frac{q^{r-1}-1}{q-1}}\sum_{\beta_{\ell,\bb}\in \cup_{\ell=1}^rB_{\ell,r}}h_{r-1}(\beta_{\ell,\bb}\cdot \zz)j(\beta_{\ell,\bb};\zz)^{\frac{q^{r-1}-1}{q-1}} \right)^q
= \left(T_{\mathfrak{p},r-1}(h_{r-1}(\zz) \right)^q.
\end{multline*}
Hence $\alpha= \mathfrak{p}^{1+q+\cdots+q^{r-2}} $, finishing the proof of the proposition.
\end{proof}

\section{Vectorial Drinfeld modular forms (VDMFs)}
In this section, we introduce vectorial Drinfeld modular forms of dimension one and dimension $r\geq 2$ separately. Our objects may be seen as an arbitrary rank analogue of functions studied in \cite{Pel12} and \cite{PP18}. 

For each $n\geq 1$, recall the affinoid subdomain $\Omega_n^r\subset \Omega^r$ defined in \S2.2. We call $\mathfrak{f}:\Omega^r\to \mathbb{T}$ \textit{a $\mathbb{T}$-valued rigid analytic function} if its restriction to each $\Omega_n^r$ is the uniform limit of rational functions in $\mathbb{T}(z_1,\dots,z_{r-1})$ having no poles in $\Omega_n^r$. We denote by $\Hol(\Omega^r,\mathbb{T})$ the set of  $\mathbb{T}$-valued rigid analytic functions on $\Omega^r$. Furthermore, we denote by $\Hol(\Omega^r,\mathbb{T}^r)$ the set of vector valued functions $\mathcal{P}:\Omega^r\to \mathbb{T}^r$ so that their each entry is a $\mathbb{T}$-valued rigid analytic function.

 Throughout the present paper, when we consider a $\mathbb{T}$-valued (vector) function, we will also emphasize the variable $t$ in the notation and write $\mathfrak{f}(\cdot,t):\Omega^r\to \mathbb{T}^d$ for some $d\geq 1$. This is mainly due to the fact that at certain cases (see for example Theorem \ref{T:main} and proof of Lemma \ref{L:weaku}), we evaluate the variable $t$ at a certain point $\mathfrak{t}\in\mathbb{C}_{\infty}$ to obtain a $\mathbb{C}_{\infty}$-valued rigid analytic function $f(\zz,\mathfrak{t}):\Omega^{r}\to \mathbb{C}_{\infty}$.

\subsection{$\mathbb{T}$-valued Drinfeld modular forms}  Let $k\in \mathbb{Z}_{\geq 0}$ and $m\in \mathbb{Z}/(q-1)\mathbb{Z}$. For any $\gamma=(a_{ij})\in \GL_r(A)$, we consider the determinant representation $\det(\cdot)^{-m}:\GL_r(A)\to \mathbb{F}_q^{\times}$ sending each $\gamma \in \GL_r(A)$ to $\det(\gamma)^{-m}$. By a slight abuse of notation, for any $\gamma\in \GL_r(K_{\infty})$ and $\mathfrak{f}\in \Hol(\Omega^r,\TT)$, we define  the $k$-th slash operator $|_{k}$ and the $k$-th slash operator $|_{k,m}$ with respect to type $m\in \mathbb{Z}/(q-1)\mathbb{Z}$  by 
\[
(\mathfrak{f}|_{k}\gamma)(\zz,t):=j(\gamma;\zz)^{-k}\mathfrak{f}(\gamma\cdot \zz,t).
\]
and 
\[
(\mathfrak{f}|_{k,m}\gamma)(\zz,t):=j(\gamma;\zz)^{-k}\det(\gamma)^{m}\mathfrak{f}(\gamma\cdot \zz,t)
\]
respectively. 
In what follows, we define 
\[
|\zz|_{\text{im}}:=\mathrm{inf}\{\inorm{z_1-\mathfrak{a}}:\mathfrak{a}=a_2z_2+\dots+a_rz_r, \ \  a_2,\dots,a_r\in K_{\infty}\}
\]
and following the notation in \cite[\S4.2]{CG22}, for $\tz\in \Omega^{r-1}$, we set 
\[
\mathfrak{I}_{\tz}:=\{ \zz=(z_1,\dots,z_r)=(z_1,\tz)\in \Omega^{r} \ \ | \ \ \inorm{z_1}=|\zz|_{\text{im}} \}.
\]
Before we define our next object, note that for any choice of $\tz\in \Omega^{r-1}$, by  \cite[(4.21)]{CG22}, we have 
\begin{equation}\label{E:ulimit}
\lim_{\substack{\zz=(z_1,\tz)\in \mathfrak{I}_{\tz}\\ |\zz|_{\infty}\to\infty}}u(\zz)=0.
\end{equation}
\begin{definition}  We call a function $\mathfrak{f}\in \Hol(\Omega^r, \mathbb{T})$ \textit{a weak $\mathbb{T}$-valued Drinfeld modular form} \textit{of weight $k$} \textit{for $\det^{-m}$} if it satisfies the following conditions:
\begin{itemize}
    \item[(i)] For all $\gamma\in \GL_r(A)$, we have
    \[
    \mathfrak{f}|_{k,m}\gamma=\mathfrak{f}.
    \]
\item[(ii)] There exists a non-negative integer $n$ such that for any choice of $\tz\in \Omega^{r-1}$, we have 
\[
\lim_{\substack{\zz=(z_1,\tz)\in \mathfrak{I}_{\tz}\\ |\zz|_{\infty}\to\infty}}u(\zz)^n\mathfrak{f}(\zz,t)<\infty.
\]
\end{itemize}
We denote by $\mathbb{WM}_{k}^{m}$ the $\mathbb{T}$-module spanned by all the weak $\mathbb{T}$-valued  Drinfeld modular forms of weight $k$ for $\det^{-m}$.
\end{definition}

\begin{lemma}[{cf. \cite[Lem. 13]{Pel12}}] \label{L:weaku} Let $\mathcal{WM}_{k,\text{mer}}^m$ be the $\mathbb{C}_{\infty}$-vector space spanned by the weak Drinfeld modular forms of weight $k$ and type $m$ which are meromorphic at infinity. Then we have 
\[
\mathbb{WM}_{k}^m=\mathcal{WM}_{k,\text{mer}}^m\otimes \mathbb{T}.
\]
\end{lemma}
\begin{proof} The proof proceeds as in the proof of \cite[Lem. 13]{Pel12}. For completeness, we provide its details. Since $\mathcal{WM}_{k,\text{mer}}^m\otimes \mathbb{T}\subseteq \mathbb{WM}_{k}^m$ is clear, we show the other direction. Let $\mathfrak{f}\in \mathbb{WM}_{k}^m$ and $\mathfrak{B}$ be the closed unit disk, that is the set of elements $\mathfrak{t}\in \mathbb{C}_{\infty}$ such that $|\mathfrak{t}|\leq 1$. Note that for any $\mathfrak{t}\in \mathfrak{B}$, the function $\mathfrak{f}(\cdot,\mathfrak{t})$ is an element in $\mathcal{WM}_{k,\text{mer}}^m$. Then we see that there exists a non-negative integer $\mu$ such that $h_r^{\mu}\mathfrak{f}(\cdot,\mathfrak{t})\in \mathcal{M}_{k+\mu((q^r-1)/(q-1))}^{m+\mu}$. Let $\{\mathfrak{b}_1,\dots,\mathfrak{b}_\nu\}$ be a $\mathbb{C}_{\infty}$-basis for $\mathcal{M}_{k+\mu((q^r-1)/(q-1))}^{m+\mu}$. Thus, there are functions $\mathfrak{c}_1,\dots,\mathfrak{c}_\nu:\mathfrak{B}\to \mathbb{C}_{\infty}$ such that for any $\zz\in \Omega^r$, we have 
\[
\mathfrak{f}(\zz,\mathfrak{t})=h_r(\zz)^{-\mu}(\mathfrak{c}_1(\mathfrak{t})\mathfrak{b}_1(\zz)+\cdots+\mathfrak{c}_{\mu}(\mathfrak{t})\mathfrak{b}_\nu(\zz)).
\]
Here, note that we are allowed to consider negative powers of $h_r$ as it is nowhere vanishing on $\Omega^r$. Now since $\mathfrak{b}_1,\dots,\mathfrak{b}_{\nu}$ are $\mathbb{C}_{\infty}$-linearly independent functions, there exist $\zz_1,\dots,\zz_\nu\in \Omega^r$ such that the matrix $D:=(h_r(\zz_i)^{-\mu}\mathfrak{b}_\ell(\zz_i))_{i\ell}\in \Mat_\nu(\mathbb{C}_{\infty})$ is invertible. Then we obtain 
\[
(\mathfrak{c}_1(\mathfrak{t}),\dots,\mathfrak{c}_\nu(\mathfrak{t}))=(\mathfrak{f}(\zz_1,\mathfrak{t}),\dots,\mathfrak{f}(\zz_\nu,\mathfrak{t}))D^{-1}.
\]
Since the above holds for any $\zz\in \Omega^r$ and $\mathfrak{f}(\cdot,t)\in \mathbb{T}$, we see that $\mathfrak{c}_1,\dots,\mathfrak{c}_\nu\in\mathbb{T}$ and hence $\mathfrak{f}\in \mathcal{WM}_{k,\text{mer}}^m\otimes \mathbb{T} $, finishing the proof of the lemma.
\end{proof}

 One can see that any element  $\mathfrak{f}\in\mathbb{WM}_{k}^{m}$ is \textit{$A$-invariant}, namely, recalling the matrix $\gamma_{a_1,\dots,a_{r-1}}\in \GL_r(A)$ defined in \eqref{E:ainv} where $(a_1,\dots,a_{r-1})\in A^{r-1}$, we have $
\mathfrak{f}(\gamma_{a_1,\dots,a_{r-1}}\cdot \zz,t)=\mathfrak{f}(\zz,t)
$
for any tuple $(a_1,\dots,a_{r-1})\in A^{r-1}$. Then, using Lemma \ref{L:weaku}, we see that  there exist a unique integer $n_0$ and a unique $\mathfrak{f}_n\in \Hol(\Omega^{r-1},\mathbb{T})$ for each $n\geq n_0$ such that 
\begin{equation}\label{E:uexp2}
\mathfrak{f}(\zz,t)=\sum_{n\geq n_0}\mathfrak{f}_n(\tz,t)u(\zz)^n
\end{equation}
whenever $\zz$ lies in some neighborhood of infinity. We call such an expansion \textit{the $u$-expansion of $\mathfrak{f}$}. We call an $A$-invariant function $\mathfrak{f}$ \textit{holomorphic at infinity} if $\mathfrak{f}_n$ is identically zero whenever $n<0$. For any $d\geq 1$, we further denote by $\Hol(\Omega^{r-1},\mathbb{T})\dbl u\dbr^{d}$ the set of vector valued functions $\mathcal{P}:\Omega^r\to \mathbb{T}^d$ so that their each entry admits a $u$-expansion as in \eqref{E:uexp2} with no polar part.

\begin{definition}  We call a weak $\mathbb{T}$-valued Drinfeld modular form $f$ \textit{a $\mathbb{T}$-valued Drinfeld modular form}  \textit{of weight $k$} \textit{for $\det^{-m}$} if $\mathfrak{f}$ is holomorphic at infinity. We denote by $\mathbb{M}_{k}^{m}$ the $\mathbb{T}$-module spanned by all the $\mathbb{T}$-valued Drinfeld modular forms of weight $k$ for $\det^{-m}$.
\end{definition}

The next corollary is a simple consequence of Lemma \ref{L:weaku}.
\begin{corollary}\label{C:tvalmodform} We have 
\[
\mathbb{M}_{k}^m=\mathcal{M}_{k}^m\otimes \mathbb{T}.
\]
\end{corollary}

\subsection{The function $\mathcal{H}_r(\cdot,t)$}

For each $\zz\in \Omega^r$, $k=q^j\in \mathbb{Z}_{\geq 1}$ with $j\geq 0$ and $1\leq i \leq r$, recall $\mathcal{E}_k^{[i]}(\zz,t)$ from \S1 and let 
\[
\mathcal{E}(\zz,t):=\begin{pmatrix}
\mathcal{E}^{[1]}_{1}(\zz,t) & \cdots &  \mathcal{E}^{[r]}_{1}(\zz,t)\\
\vdots & & \vdots \\
\mathcal{E}^{[1]}_{q^{r-1}}(\zz,t) & \cdots &  \mathcal{E}^{[r]}_{q^{r-1}}(\zz,t)
\end{pmatrix}\in \Mat_{r}(\mathbb{T}).
\]
Let $\zz\in \Omega^r$. Recall the Drinfeld module $\phi^{\zz}$ associated to the $A$-lattice $A^r\zz$ from Example \ref{Ex:1}. Following the notation in \cite[\S5.2]{CG22}, we let $s_i(\zz,t):=s_{\phi^{\zz}}(z_i,t)
$
and consider the matrix
\[
\mathcal{F}(\zz,t):=\begin{pmatrix}
s_1(\zz,t)&\dots &s_1^{(r-1)}(\zz,t)\\
\vdots & & \vdots \\
s_r(\zz,t)&\dots &s_r^{(r-1)}(\zz,t)
\end{pmatrix}\in \Mat_{r}(\TT)
\]
where $s_i^{(j)}(\zz,t)$ is $s_i(\zz,t)$ under the $j$-th twisting operator defined in \eqref{E:twisting}. By \cite[Prop. 3.4]{CG21}, we know that 
\begin{equation}\label{E:det1}
\det(\mathcal{F}(\zz,t))=\frac{\tilde{\pi}^{\frac{1-q^r}{q-1}}\omega(t)}{h_r(\zz)}\in \mathbb{T}^{\times}.
\end{equation}
Hence $\mathcal{F}(\zz,t)\in \GL_r(\mathbb{T})$. Now, for each $\ell\geq 0$, recall the exponential and logarithm coefficients $\alpha_\ell(\zz)$ and $\beta_\ell(\zz)$ respectively corresponding to the Drinfeld module $\phi^{\zz}$ from Example \ref{Ex:1}. Then we define 
	\[
	c_{q^\ell-1}(\zz,t):=\frac{\beta_\ell(\zz)}{\theta-t}+\frac{\alpha_1(\zz)\beta_{\ell-1}(\zz)^q}{\theta^q-t}+\dots+ \frac{\alpha_{\ell-1}(\zz)\beta_{1}(\zz)^{q^{\ell-1}}}{\theta^{q^{\ell-1}}-t}+\frac{\alpha_\ell(\zz)}{\theta^{q^\ell}-t}\in \mathbb{T}.
	\]
In \cite[Lem. 5.16]{CG22}, Chen and the first author showed that 
\begin{equation}\label{E:mainid}
\mathcal{E}(\zz,t)\mathcal{F}(\zz,t)=-\begin{pmatrix}
	c_0(\zz,t)&  & &\\
	c_{q-1}(\zz,t)&c_0(\zz,t)^{(1)} & & \\
	\vdots& &\ddots & \\
	c_{q^{r-1}-1}(\zz,t)&c_{q^{r-2}-1}(\zz,t)^{(1)} & \dots & c_0(\zz,t)^{(r-1)}
	\end{pmatrix}.
    \end{equation}
   Since $c_0(\zz,t)=1/(\theta-t)$, using \eqref{E:omegafunceq}, \eqref{E:det1} and \eqref{E:mainid}, we see that 
\begin{equation}\label{E:det2}
    \det(\mathcal{E}(\zz,t))=\frac{\tilde{\pi}^{\frac{q^r-1}{q-1}}h_r(\zz)}{(t-\theta)\cdots (t-\theta^{q^{r-1}})\omega(t)}=\frac{\tilde{\pi}^{\frac{q^r-1}{q-1}}h_r(\zz)}{\omega(t)^{(r)}}.
\end{equation}

In what follows, we consider the function $\mathcal{H}_r(\cdot,t):\Omega^r\to \mathbb{T}$ given by 
\[
\mathcal{H}_r(\zz,t):=\det \begin{pmatrix}
    \mathcal{E}^{[1]}_{1}(\zz,t) & \cdots &  \mathcal{E}^{[r-1]}_{1}(\zz,t)\\
\vdots & & \vdots \\
\mathcal{E}^{[1]}_{q^{r-2}}(\zz,t) & \cdots &  \mathcal{E}^{[r-1]}_{q^{r-2}}(\zz,t)
\end{pmatrix}
\]
which is the $(r,r)$-cofactor of $\mathcal{E}(\zz,t)$. Our first main result may be seen as a generalization of \cite[Thm. 3.19]{PP18}. 
\begin{theorem}\label{T:funcH} The following statements hold.
\begin{itemize}
\item[(i)] We have 
\[
\mathcal{H}_r(\zz,t)=\frac{\tilde{\pi}^{\frac{q^r-1}{q-1}}h_r(\zz)}{(t-\theta)\cdots (t-\theta^{q^{r-2}})\omega(t)}s_r^{(r-1)}(\zz,t)=\frac{\tilde{\pi}^{\frac{q^r-1}{q-1}}h_r(\zz)}{\omega(t)^{(r-1)}}s_r^{(r-1)}(\zz,t).
\]
Moreover, for any fixed $\zz\in \Omega^r$, $\mathcal{H}_r(\zz,\cdot)$ may be extended to an entire function of the variable $t$.
\item[(ii)] Recall the elements $D_\ell\in A$ defined in Example \ref{Ex:C} for each $\ell\in \mathbb{Z}_{\geq 0}$. For any $n\geq r-1$, we have
\[
\mathcal{H}_r(\zz,\theta^{q^{n}})=\tilde{\pi}^{\frac{q^r-1}{q-1}-q^{n}}D_{n-(r-1)}^{q^{r-1}}h_r(\zz)\alpha_{n-(r-1)}(\zz)^{q^{r-1}}.
\]
In particular, each $\mathcal{H}_r(\zz,\theta^{q^n})\in \mathcal{M}_{\frac{q^{r-1}-1}{q-1}+q^n}^{1}$ and 
\[
\mathcal{H}_r(\zz,\theta^{q^{r-1}})=\tilde{\pi}^{\frac{q^{r-1}-1}{q-1}}h_r(\zz).
\]
Furthermore, each $ \mathcal{H}_r(\zz,\theta^{q^n}) $ is a Drinfeld cusp form.
\end{itemize}
\end{theorem}
\begin{proof} By  \eqref{E:mainid}, we have 
\begin{multline}\label{E:mainid2}
\det(\mathcal{E}(\zz,t))\mathcal{F}(\zz,t)\\=-\det(\mathcal{E}(\zz,t))\mathcal{E}(\zz,t)^{-1}\begin{pmatrix}
	c_0(\zz,t)&  & &\\
	c_{q-1}(\zz,t)&c_0(\zz,t)^{(1)} & & \\
	\vdots& &\ddots & \\
	c_{q^{r-1}-1}(\zz,t)&c_{q^{r-2}-1}(\zz,t)^{(1)} & \dots & c_0(\zz,t)^{(r-1)}
	\end{pmatrix}.
\end{multline}
Since $\mathcal{H}_r(\zz,t)$ is the $(r,r)$-cofactor of $\mathcal{E}(\zz,t)$, comparing the $(r,r)$-th entry of both sides of \eqref{E:mainid2} and using \eqref{E:det2}, we have 
\begin{multline}\label{E:hfuncobtain}
\frac{\tilde{\pi}^{\frac{q^r-1}{q-1}}h_r(\zz)}{\omega(t)^{(r)}}s_r^{(r-1)}(\zz,t)=\frac{\tilde{\pi}^{\frac{q^r-1}{q-1}}h_r(\zz)}{(t-\theta)\cdots (t-\theta^{q^{r-1}})\omega(t)}s_r^{(r-1)}(\zz,t)\\
=\frac{\tilde{\pi}^{\frac{q^r-1}{q-1}}h_r(\zz)}{(t-\theta^{q^{r-1}})\omega(t)^{(r-1)}}s_r^{(r-1)}(\zz,t)
=-\mathcal{H}_{r}(\zz,t)c_0^{(r-1)}(\zz,t)=\frac{\mathcal{H}_r(\zz,t)}{t-\theta^{q^{r-1}}}.
\end{multline}
Thus, by \eqref{E:hfuncobtain}, we obtain the first assertion. The second part of the first assertion and the first assertion of part (ii) follow from the fact that $s_r^{(r-1)}(\zz,t)$ for any fixed $\zz\in \Omega^r$ and $\omega^{(r-1)}(t)$ have simple poles at $t=\theta^{q^n}$ for $n\geq r-1$ with residues $-\alpha_{n-(r-1)}(\zz)^{q^{r-1}}$ and $-\tilde{\pi}^{q^{n}}D_{n-(r-1)}^{-q^{r-1}}$ respectively and no other poles (Proposition \ref{P:and}). The second assertion of part (ii) is a simple consequence of the first assertion of part (ii) and finally, since each $ \mathcal{H}_r(\zz,\theta^{q^n}) $ has non-zero type, the last assertion follows from \cite[Thm. 17.6]{BBP18}.
\end{proof}

Let $\gamma=(a_{ij})_{i,j}\in \GL_r(A)$.  By the proof of \cite[Thm. 5.5]{CG21}, for $1\leq i \leq r$, we have 
\begin{equation}\label{E:andgenfunceq}
s_i(\gamma\cdot \zz,t)=j(\gamma;\zz)^{-1}(a_{i1}(t)s_1(\zz,t)+\cdots+a_{ir}(t)s_r(\zz,t)).
\end{equation}
Furthermore, let $a_1,\dots,a_{r-1}\in A$ and recall the matrix $\gamma_{a_1,\dots,a_{r-1}}\in \GL_r(A)$ defined in \eqref{E:ainv}. Then, by \eqref{E:andgenfunceq}, for any $\zz\in \Omega^r$, we see that 
\[
s_r^{(r-1)}(\gamma_{a_1,\dots,a_{r-1}}\cdot \zz,t)=s_r^{(r-1)}(\zz,t).
\]
Thus, Theorem \ref{T:funcH}(i) implies
\[
\mathcal{H}_r(\gamma_{a_1,\dots,a_{r-1}}\cdot \zz,t)=\mathcal{H}_r(\zz,t)
\]
and hence, $\mathcal{H}_r(\cdot,t)$ is an $A$-invariant $\mathbb{T}$-valued rigid analytic function. Thus, it has a $u$-expansion given by
\[
\mathcal{H}_r(\cdot,t)=\sum_{n\in \mathbb{Z}}\mathfrak{f}_n(\cdot,t)u^n
\]
where, for each $n\in \mathbb{Z}$, $\mathfrak{f}_n:\Omega^{r-1}\to \mathbb{T}$ is a uniquely defined rigid analytic function. In what follows, we take a closer look at this $u$-expansion. Let
\[
h_r=\sum_{\ell=1}^{\infty}h_{\ell,r}u^{\ell}
\]
be the $u$-expansion of $h_r$ and for each $n\geq 0$ consider the $u$-expansion of $\alpha_{n}$ given by
\[
\alpha_n=\sum_{i=0}^{\infty}\alpha_{n,i}u^i.
\]
Then by Proposition \ref{P:and}(i) and Theorem \ref{T:funcH}(i), we see that the $u$-expansion of $\mathcal{H}_r(\cdot,t)$ can be explicitly given by 
\[
\mathcal{H}_r(\zz,t)=\frac{\tilde{\pi}^{\frac{q^r-1}{q-1}}h_r(\zz)}{\omega(t)^{(r-1)}}s_r^{(r-1)}(\zz,t)=\sum_{m =1}^{\infty}\left(\frac{\tilde{\pi}^{\frac{q^r-1}{q-1}}}{\omega(t)^{(r-1)}}\sum_{\ell+jq^{r-1}=m}h_{\ell,r}(\tz)\sum_{n= 0}^{\infty}\frac{\alpha_{n,j}(\tz)^{q^{r-1}}}{\theta^{q^{n+r-1}}-t}\right)u(\zz)^m
\]
whenever $\zz=(z_1,\zz)\in \Omega^r$ is sufficiently large. Thus, the $u$-expansion of $\mathcal{H}_r(\cdot,t)$ neither has polar part nor constant term and hence is an element in  $u\Hol(\Omega^{r-1},\mathbb{T})\dbl u\dbr$. In our next lemma, we will explicitly describe the first coefficient of the $u$-expansion of $\mathcal{H}_r(\cdot,t)$.
\begin{lemma}\label{L:uexp}
Let $\mathcal{H}_r(\cdot,t)=\sum_{n=1}^{\infty}\mathfrak{f}_n(\cdot,t)u^n$ be the $u$-expansion of $\mathcal{H}_r(\cdot,t)$. Then 
\[
\mathfrak{f}_1(\cdot,t)=(-1)^r\tilde{\pi}\mathcal{H}_{r-1}^{(1)}(\cdot,t).
\]
\end{lemma}
\begin{proof} Let $\zz=(z_1,\zz)\in \Omega^r$. By \cite[Prop. 15.3]{BBP18}, for each $n\geq 0$, we have $\alpha_{n,0}(\tz)=\alpha_{n}(\tz)$. Here, by $\alpha_n(\tz)$, we mean the $n$-th coefficient of the exponential series of the Drinfeld module $\phi^{\tz}$ corresponding to the $A$-lattice $\Lambda_{\tz}$. Moreover, by \eqref{E:huexp}, we have $h_{1,r}(\tz)=(-1)^r h_{r-1}(\tz)^q$. Thus, we obtain    
    \begin{multline*}
\mathfrak{f}_1(\tz,t)=\frac{\tilde{\pi}^{\frac{q^r-1}{q-1}}h_{1,r}(\tz)}{\omega(t)^{(r-1)}}\sum_{n= 0}^{\infty}\frac{\alpha_{n,0}(\tz)^{q^{r-1}}}{\theta^{q^{n+r-1}}-t}
=\frac{\tilde{\pi}^{\frac{q^r-1}{q-1}}(-1)^r h_{r-1}(\tz)^q}{\omega(t)^{(r-1)}}\sum_{n=0}^{\infty}\frac{\alpha_n(\tz)^{q^{r-1}}}{\theta^{q^{n+r-1}}-t}\\
=(-1)^r\tilde{\pi}\left(\frac{\tilde{\pi}^{\frac{q^{r-1}-1}{q-1}}h_{r-1}(\tz)}{\omega(t)^{(r-2)}}\sum_{n=0}^{\infty}\frac{\alpha_n(\tz)^{q^{r-2}}}{\theta^{q^{n+r-2}}-t}\right)^{(1)}
=(-1)^r\tilde{\pi}\left(\frac{\tilde{\pi}^{\frac{q^{r-1}-1}{q-1}}h_{r-1}(\tz)}{\omega(t)^{(r-2)}}s_{r-1}^{(r-2)}(\tz,t)\right)^{(1)}\\=
(-1)^r\tilde{\pi}\mathcal{H}_{r-1}^{(1)}(\tz,t)
    \end{multline*}
where the fourth equality follows from Proposition \ref{P:and}(i) and the last equality follows from Theorem \ref{T:funcH}(i). Thus, it finishes the proof of the lemma.
\end{proof}

We finish this subsection with a crucial analysis on Anderson generating functions which will be useful throughout the paper.
\begin{lemma}\label{L:Cor516} Let $0\leq i\leq r-1$,  $1\leq j \leq r$ and $\tilde{g}_{r-1}:\Omega^{r-1}\to \mathbb{C}_{\infty}$ be the discriminant function of rank $r-1$. For each $\tz\in \Omega^{r-1}$, we have
\[
\lim_{\substack{\zz=(z_1,\tz)\in \mathfrak{I}_{\tz}\\ |\zz|_{\infty}\to\infty}}u(\zz)s_j^{(i)}(\zz,t)=\begin{cases}
\tilde{\pi}^{-1}\tilde{g}_{r-1}(\tz)^{-1} &\text{ if } i=r-1 \text{ and } j=1\\
0 &  \text{otherwise.}\\
\end{cases}
\] 
\end{lemma}
\begin{proof} Recall that
\[
s_j^{(i)}(\zz,t)=\sum_{n=0}^{\infty}\exp_{\phi^{\zz}}\left(\frac{z_j}{\theta^{n+1}}\right)^{q^i}t^n.
\]
By \cite[Rem. 3.22]{CG21}, note that for each $n\geq 0$,  we have
\begin{equation}\label{E:limit11}
\lim_{\substack{\zz=(z_1,\tz)\in \mathfrak{I}_{\tz}\\ |\zz|_{\infty}\to\infty}}\exp_{\phi^{\zz}}\left(\frac{z_j}{\theta^{n+1}}\right)^{q^i}=\lim_{\substack{\zz=(z_1,\tz)\in \mathfrak{I}_{\tz}\\ |\zz|_{\infty}\to\infty}}\exp_{\phi^{\tz}}\left(\frac{z_j}{\theta^{n+1}}\right)^{q^i}.
\end{equation}
Then \eqref{E:limit11} implies
\begin{equation}\label{E:limitgen}
\lim_{\substack{\zz=(z_1,\tz)\in \mathfrak{I}_{\tz}\\ |\zz|_{\infty}\to\infty}}u(\zz)\exp_{\phi^{\zz}}\left(\frac{z_j}{\theta^{n+1}}\right)^{q^i}=\lim_{\substack{\zz=(z_1,\tz)\in \mathfrak{I}_{\tz}\\ |\zz|_{\infty}\to\infty}}\frac{\tilde{\pi}^{-1}\exp_{\phi^{\zz}}\left(\frac{z_j}{\theta^{n+1}}\right)^{q^i}}{\exp_{\phi^{\tz}}(z_1)}=\lim_{\substack{\zz=(z_1,\tz)\in \mathfrak{I}_{\tz}\\ |\zz|_{\infty}\to\infty}}\frac{\tilde{\pi}^{-1}\exp_{\phi^{\tz}}\left(\frac{z_j}{\theta^{n+1}}\right)^{q^i}}{\exp_{\phi^{\tz}}(z_1)}.
\end{equation}
Let $2\leq j \leq r$. Since, in this case, $\exp_{\phi^{\tz}}\left(\frac{z_j}{\theta^{n+1}}\right)$ is independent of $z_1$, by \eqref{E:ulimit} and \eqref{E:limit11}, we easily see that
\begin{equation}\label{E:limit12}
\lim_{\substack{\zz=(z_1,\tz)\in \mathfrak{I}_{\tz}\\ |\zz|_{\infty}\to\infty}}u(\zz)\exp_{\phi^{\zz}}\left(\frac{z_j}{\theta^{n+1}}\right)^{q^i}=0
\end{equation}
for any $n\geq 0$. Hence we obtain
\begin{multline}\label{E:con1}
\lim_{\substack{\zz=(z_1,\tz)\in \mathfrak{I}_{\tz}\\ |\zz|_{\infty}\to\infty}}u(\zz)s_j^{(i)}(\zz,t)=\lim_{\substack{\zz=(z_1,\tz)\in \mathfrak{I}_{\tz}\\ |\zz|_{\infty}\to\infty}}\sum_{n=0}^{\infty}u(\zz)\exp_{\phi^{\zz}}\left(\frac{z_j}{\theta^{n+1}}\right)^{q^i}t^n\\
=\sum_{n=0}^{\infty}
\left(\lim_{\substack{\zz=(z_1,\tz)\in \mathfrak{I}_{\tz}\\ |\zz|_{\infty}\to\infty}}u(\zz)\exp_{\phi^{\zz}}\left(\frac{z_j}{\theta^{n+1}}\right)^{q^i}\right)t^n=0.
\end{multline}
Now we let $j=1$. First, we observe that $(z_1,\dots,z_r)\in \mathfrak{I}_{\tz}$
 if and only if $(z_1/\theta^{n+1},z_2,\dots,z_r)\in \mathfrak{I}_{\tz}$. Therefore, again by \cite[Rem. 3.22]{CG21}, we have
 \[
 \left|\exp_{\phi^{\tz}}\left(\frac{z_1}{\theta^{n+1}}\right)\right|=\left|\frac{z_1}{\theta^{n+1}}\right|\prod_{\substack{\lambda \in \Lambda_{\tz}\\ 0<\lambda<z_1/\theta^{n+1}}}\frac{|z_1|}{|\theta^{n+1}\lambda|}.
 \]
 In other words, 
\begin{equation}\label{E:onecasenorm}
\left|\exp_{\phi^{\tz}}\left(\frac{z_1}{\theta^{n+1}}\right)\right|\to \infty \text{ as } |z_1|\to \infty.
\end{equation}
By \eqref{E:funceqtn}, we further note that there exist $\mathfrak{y}_1,\dots,\mathfrak{y}_{(n+1)(r-1)}\in \mathbb{C}_{\infty}$ with $\mathfrak{y}_{(n+1)(r-1)}\in \mathbb{C}_{\infty}^{\times}$ which are all independent of $z_1$ such that 
 \begin{multline}\label{E:onecase}
u(\zz)\exp_{\phi^{\tz}}\left(\frac{z_1}{\theta^{n+1}}\right)^{q^i}=\frac{\tilde{\pi}^{-1}\exp_{\phi^{\tz}}\left(\frac{z_1}{\theta^{n+1}}\right)^{q^i}}{\exp_{\phi^{\tz}}(z_1)}=\frac{\tilde{\pi}^{-1}\exp_{\phi^{\tz}}\left(\frac{z_1}{\theta^{n+1}}\right)^{q^i}}{\phi^{\tz}_{\theta^{n+1}}\left(\exp_{\phi^{\tz}}\left(\frac{z_1}{\theta^{n+1}}\right)\right)}\\=\frac{\tilde{\pi}^{-1}\exp_{\phi^{\tz}}\left(\frac{z_1}{\theta^{n+1}}\right)^{q^i}}{\theta^{n+1}\exp_{\phi^{\tz}}\left(\frac{z_1}{\theta^{n+1}}\right)+\mathfrak{y}_1\exp_{\phi^{\tz}}\left(\frac{z_1}{\theta^{n+1}}\right)^q+\cdots+\mathfrak{y}_{(n+1)(r-1)}\exp_{\phi^{\tz}}\left(\frac{z_1}{\theta^{n+1}}\right)^{q^{(n+1)(r-1)}}}.
 \end{multline}
 Note that $\mathfrak{y}_{(n+1)(r-1)}=\tilde{g}_{r-1}(\tz)$ if $n=0$. Then \eqref{E:limitgen}, \eqref{E:onecasenorm} and \eqref{E:onecase} imply that 
\begin{multline}\label{E:limit14}
\lim_{\substack{\zz=(z_1,\tz)\in \mathfrak{I}_{\tz}\\ |\zz|_{\infty}\to\infty}}u(\zz)\exp_{\phi^{\zz}}\left(\frac{z_1}{\theta^{n+1}}\right)^{q^i}=\lim_{\substack{\zz=(z_1,\tz)\in \mathfrak{I}_{\tz}\\ |\zz|_{\infty}\to\infty}}u(\zz)\exp_{\phi^{\tz}}\left(\frac{z_1}{\theta^{n+1}}\right)^{q^i}\\
=\begin{cases}
\tilde{\pi}^{-1}\tilde{g}_{r-1}(\tz)^{-1} &\text{ if } i=r-1 \text{ and } n=0\\
0 &  \text{otherwise.}\\
\end{cases}
\end{multline}
Finally, by \eqref{E:limit14}, we obtain 
\begin{multline}\label{E:con2}
\lim_{\substack{\zz=(z_1,\tz)\in \mathfrak{I}_{\tz}\\ |\zz|_{\infty}\to\infty}}u(\zz)s_1^{(i)}(\zz,t)=\lim_{\substack{\zz=(z_1,\tz)\in \mathfrak{I}_{\tz}\\ |\zz|_{\infty}\to\infty}}\sum_{n=0}^{\infty}u(\zz)\exp_{\phi^{\zz}}\left(\frac{z_1}{\theta^{n+1}}\right)^{q^i}t^n\\
=\sum_{n=0}^{\infty}
\left(\lim_{\substack{\zz=(z_1,\tz)\in \mathfrak{I}_{\tz}\\ |\zz|_{\infty}\to\infty}}u(\zz)\exp_{\phi^{\zz}}\left(\frac{z_1}{\theta^{n+1}}\right)^{q^i}\right)t^n=\begin{cases}
\tilde{\pi}^{-1}\tilde{g}_{r-1}(\tz)^{-1} &\text{ if } i=r-1\\
0 &  \text{otherwise.}\\
\end{cases}
\end{multline}
Thus, \eqref{E:con1} and \eqref{E:con2} finish the proof of the lemma.
\end{proof}

\subsection{Vectorial Drinfeld modular forms for $\rho_m$} Let $m\in \mathbb{Z}/(q-1)\mathbb{Z}$,  $m\neq 0$ for $q>2$ and $m=1$ for $q=2$. For any $\gamma=(a_{ij})\in \GL_r(A)$, we set $\overline{\gamma}:=(a_{ij}(t))\in \GL_r(\mathbb{F}_q[t])$. We now consider the map $\rho_m:\GL_r(A)\to \GL_r(\mathbb{F}_q[t])$ given by 
\[
\rho_m(\gamma):=\det(\overline{\gamma})^{-m}\overline{\gamma}.
\]
For simplicity, we let $\rho:=\rho_1$. 
\begin{definition} Let $k$ be an integer. We call a function $\mathcal{P}=(\mathcal{P}_1,\dots,\mathcal{P}_r)^{\tr}\in \Hol(\Omega^r, \mathbb{T}^r)$  \textit{a weak vectorial Drinfeld modular form} (weak VDMF) \textit{of weight $k$} \textit{with respect to $\rho_m$} if it satisfies the following conditions: 
\begin{itemize}
    \item[(i)] For all $\gamma\in \GL_r(A)$, we have
    \[
    \mathcal{P}(\gamma\cdot \zz)=j(\gamma;\zz)^k\rho_m(\gamma)\mathcal{P}(\zz).
    \]
\item[(ii)] For each $1\leq i \leq r$, there exists a non-negative integer $n_i$ such that for any choice of $\tz\in \Omega^{r-1}$, we have 
\[
\lim_{\substack{\zz=(z_1,\tz)\in \mathfrak{I}_{\tz}\\ |\zz|_{\infty}\to\infty}}u(\zz)^{n_i}\mathcal{P}_i(\zz,t)<\infty.
\]    
\end{itemize}
We denote by $\mathbb{WM}_{k}(\rho_m)$ the $\mathbb{T}$-module of weak VDMFs of weight $k$ with respect to $\rho_m$.    \end{definition}

Next we introduce vectorial Drinfeld modular forms.
\begin{definition}\label{D:vdmf} We call a weak VDMF  $\mathcal{P}\in \Hol(\Omega^r, \mathbb{T}^r)$ \textit{a vectorial Drinfeld modular form} (VDMF) \textit{of weight $k$} \textit{with respect to $\rho_m$} if for any $\tz\in \Omega^{r-1}$, letting $\Upsilon:=\begin{pmatrix}
    u &  & & \\
     & 1 & & \\
     & & \ddots &  \\ 
     & & & 1
\end{pmatrix}$,  we have 
\[
(\Upsilon \mathcal{P})(\zz)\to \begin{pmatrix}
    0\\
    \vdots\\
    0
\end{pmatrix}
\]
as $|\zz|_{\infty}\to \infty$ where $\zz=(z_1,\tz) \in \mathfrak{I}_{\tz}$. We denote by $\mathbb{M}_{k}(\rho_m)$ the $\mathbb{T}$-module of VDMFs of weight $k$ with respect to $\rho_m$.
    \end{definition}

In what follows, we introduce some examples of VDMFs. For each $1\leq i\leq r$, we  consider the function $\mathcal{G}_i(\cdot,t):\Omega^r\to \mathbb{T}^r$ given by 
\[
\mathcal{G}_i(\zz,t):=\frac{\tilde{\pi}^{\frac{q^r-1}{q-1}}h_r(\zz)}{(t-\theta)\cdots (t-\theta^{q^{r-2}})\omega(t)}\begin{pmatrix}
    s_1^{(i-1)}(\zz,t)\\
    \vdots\\
    s_r^{(i-1)}(\zz,t)
\end{pmatrix}=\frac{\tilde{\pi}^{\frac{q^r-1}{q-1}}h_r(\zz)}{\omega(t)^{(r-1)}}\begin{pmatrix}
    s_1^{(i-1)}(\zz,t)\\
    \vdots\\
    s_r^{(i-1)}(\zz,t)
\end{pmatrix}.
\]

\begin{lemma}\label{L:funcg} We have $\mathcal{G}_i\in \Hol(\Omega^r,\mathbb{T}^r)$. Moreover, for any fixed $\zz\in \Omega^r$  and $1\leq i \leq r-1$, each entry of $\mathcal{G}_i(\zz,\cdot)$ may be extended to a meromorphic function of $t$ with simple poles at $t=\theta^{q^{i-1}},\dots,\theta^{q^{r-2}}$. Moreover, for any fixed $\zz\in \Omega^r$, each entry of $\mathcal{G}_{r}(\zz,t)$ may be extended to an entire function of $t$. Furthermore, the last entry of $\mathcal{G}_r(\cdot,t)$ is equal to $\mathcal{H}_r(\cdot,t)$.
\end{lemma}
\begin{proof}
Let $1\leq j \leq r$. Since, by Proposition \ref{P:and}(i), for a fixed $\mathfrak{t}$ in the algebraic closure of $\mathbb{F}_q$, $s_j(\zz,\mathfrak{t})=s_{\phi^{\zz}}(z_j,\mathfrak{t})$ is a holomorphic function, the first assertion follows. By Proposition \ref{P:and}(ii), we note that  $s_j^{(i-1)}(\zz,t)$ has simple poles at $t=\theta^{q^n}$ for $n\geq i-1$ and no other poles. On the other hand, by the product expansion of $\omega(t)$ given in \eqref{D:omega}, we see that  $(\omega(t)^{(r-1)})^{-1}=(t-\theta)^{-1}\cdots (t-\theta^{q^{r-2}})^{-1}\omega(t)^{-1}$ has zeros at $t=\theta^{q^m}$ for $m\geq r-1$. Hence, for $1\leq i \leq r-1$, $\mathcal{G}_i(\zz,\cdot)$ has a pole at $t=\theta^{q^n}$ for each $i-1\leq n\leq r-2$ and no other poles. The same discussion also implies that each entry of $\mathcal{G}_r(\zz,t)$ is an entire function of $t$. Finally, the last assertion immediately follows from Theorem \ref{T:funcH}(i).
\end{proof}

\begin{proposition} We have $\mathcal{G}_i(\cdot,t)\in \mathbb{M}_{\frac{q^r-1}{q-1}-q^{i-1}}(\rho)$. 
\end{proposition}
\begin{proof} Let $\gamma=(a_{ij})\in \GL_r(A)$. Since $h_r\in \mathcal{M}_{(q^r-1)/(q-1)}^{1}$, applying the twisting operator defined in \eqref{E:twisting}, \eqref{E:andgenfunceq} implies
    \[
    \mathcal{G}_i(\gamma\cdot \zz)=j(\gamma;\zz)^{\frac{q^r-1}{q-1}-q^{i-1}}\rho(\gamma)\mathcal{G}_i(\zz).
    \]
    On the other hand, since $h_r$ is a Drinfeld cusp form, by using \eqref{E:ulimit} and Lemma \ref{L:Cor516}, we see that  for any choice of $\tz\in \Omega^{r-1}$,
$
(\Upsilon \mathcal{G}_i)(\zz)\to (0,\dots,0)^{\tr}$
 as $|\zz|_{\infty}\to \infty$ where $\zz=(z_1,\tz) \in \mathfrak{I}_{\tz}$ and hence $\mathcal{G}_i(\cdot,t)\in \mathbb{M}_{\frac{q^r-1}{q-1}-q^{i-1}}(\rho)$ as desired.
\end{proof}

Next, generalizing \cite[Thm. 4.22(1,2)]{PP18}, we obtain an equivalent condition for a weak VDMF to be a VDMF. We let $X_1,\dots,X_r$ be indeterminates over $\mathbb{T}$. Then for any $z\in \mathbb{C}_{\infty}$, we set
\[
\mathcal{D}(X_1,\dots,X_r,z):=\begin{pmatrix}
X_1&X_2&\cdots & X_r\\
s_{\phi^{\tz}}(z,t)  & s_2(\tz,t) & \cdots & s_r(\tz,t)\\
\vdots & \vdots & &  \vdots\\
s_{\phi^{\tz}}^{(r-2)}(z,t)& s_2^{(r-2)}(\tz,t) & \cdots  & s_r^{(r-2)}(\tz,t)
\end{pmatrix}\in \Mat_{r}(\mathbb{T}[X_1,\dots,X_r]).
\]

For each $1\leq \mu \leq r-1$, we define the function $\chi_{\mu,\tz}(\cdot,t):\Omega\to \mathbb{T}$ by 
\[
\chi_{\mu,\tz}(z,t):=(-1)^{\mu-1}\frac{\tilde{\pi}^{\frac{q^{r-1}-1}{q-1}} h_{r-1}(\tz)}{\omega(t)}\times \text{$(1,\mu+1)$-minor of $\mathcal{D}(X_1,\dots,X_r,z)$}.
\]
\begin{remark} Let $\tilde{C}$ be the Drinfeld module of rank one corresponding to the $A$-lattice generated by $1$. We emphasize that when $r=2$, the function $\chi_{1,\tz}(\cdot,t)=\tilde{\pi}s_{\tilde{C}}(\cdot,t)/\omega(t)$, after a certain normalization, was already studied in \cite[\S2.3]{PP18}.
\end{remark}

In our next lemma, we state an important growth property of $\chi_{\mu,\tz}$.
\begin{lemma}\label{L:growth} For any choice of $\tz \in \Omega^{r-1}$ and $1\leq \mu \leq r-1$, we have 
\[
 \lim_{\substack{\zz=(z_1,\tz)\in \mathfrak{I}_{\tz}\\ |\zz|_{\infty}\to\infty}} u(\zz)\chi_{\mu,\tz}(z_1,t)=0.
\]

\end{lemma}
\begin{proof} By \eqref{E:limit14}, for $0\leq i \leq r-2$, we have
    \begin{equation}\label{E:lim1}
 \lim_{\substack{\zz=(z_1,\tz)\in \mathfrak{I}_{\tz}\\ |\zz|_{\infty}\to\infty}} u(\zz)s_{\phi^{\tz}}^{(i)}(z_1,t)=0
\end{equation}
and, since, for $2\leq j \leq r-1$, $s_{\phi^{\tz}}(z_j,t)$ does not depend on $z_1$, by \eqref{E:ulimit}, we obtain
 \begin{equation}\label{E:lim3}
 \lim_{\substack{\zz=(z_1,\tz)\in \mathfrak{I}_{\tz}\\ |\zz|_{\infty}\to\infty}} u(\zz)s_{\phi^{\tz}}^{(i)}(z_j,t)= \lim_{\substack{\zz=(z_1,\tz)\in \mathfrak{I}_{\tz}\\ |\zz|_{\infty}\to\infty}} u(\zz)s_j^{(i)}(\tz,t)=0.
\end{equation}
Thus, the lemma follows from \eqref{E:lim1}, \eqref{E:lim3} and the definition of $\chi_{\mu,\tz}$.
\end{proof}

Let $z\in \mathbb{C}_{\infty}$. We now define the matrix $\Theta_{t,\tz}(z,t)$ by 
\[
\Theta_{t,\tz}(z,t):=\begin{pmatrix}
    1 & \chi_{1,\tz}(z,t) & \cdots & \chi_{r-1,\tz}(z,t)\\
     & \ddots & & \\
     & & \ddots & \\
     & & & 1
\end{pmatrix}.
\]
\begin{lemma}\label{P:funceq}  Let $\tz\in \Omega^{r-1}$ and $z\in \mathbb{C}_{\infty}$. For any $a_1,\dots,a_{r-1}\in A$ and each $1\leq \mu \leq r$, we have 
\[
\chi_{\mu,\tz}(z+a_1z_2+\dots+a_{r-1}z_r,t)=\chi_{\mu,\tz}(z,t)+a_{\mu}(t).
\]
    In particular, 
    \[
    \Theta_{t,\tz}(a_1z_2+\dots+a_{r-1}z_r,t)=\begin{pmatrix}
    1 & a_{1}(t) & \cdots & a_{r-1}(t)\\
     & \ddots & & \\
     & & \ddots & \\
     & & & 1
\end{pmatrix}.
\]
    \end{lemma}
\begin{proof}
Observe that
\begin{multline*}
s_{\phi^{\tz}}(z+a_1z_2+\dots+a_{r-1}z_r,t)=\sum_{i=0}^{\infty}\exp_{\phi^{\tz}}\left(\frac{z+a_1z_2+\dots+a_{r-1}z_r}{\theta^{i+1}}\right)t^i\\
=\sum_{i=0}^{\infty}\exp_{\phi^{\tz}}\left(\frac{z}{\theta^{i+1}}\right)t^i+\sum_{i=0}^{\infty}\exp_{\phi^{\tz}}\left(\frac{a_1z_2}{\theta^{i+1}}\right)t^i+\cdots+\sum_{i=0}^{\infty}\exp_{\phi^{\tz}}\left(\frac{a_{r-1}z_r}{\theta^{i+1}}\right)t^i\\
=s_{\phi^{\tz}}(z,t)+a_1(t)s_2(\tz,t)+\dots+a_{r-1}(t)s_r(\tz,t).
\end{multline*}
Thus, the lemma follows from the definition of $\chi_m(\cdot,t)$, sum, switching columns and the scalar multiple property of determinants as well as the fact that 
\[
\det \begin{pmatrix} s_2(\tz,t) & \cdots & s_r(\tz,t)\\
 \vdots &  &   \vdots\\
 s_2^{(r-2)}(\tz,t) & \cdots  & s_r^{(r-2)}(\tz,t)
\end{pmatrix}=\tilde{\pi}^{\frac{1-q^{r-1}}{q-1}}h_{r-1}(\tz)^{-1}\omega(t)
\]
which follows from \eqref{E:det1}. The second assertion immediately follows from the first assertion.
\end{proof}
For any $\mathcal{P}\in \Hol(\Omega^{r-1},\mathbb{T}^r)$ and $1\leq i \leq r-1$, we define the rigid analytic function $\chi_{i,\tz}\mathcal{P}(\cdot, t):\Omega^{r}\to \mathbb{T}^r$ by 
\[
\chi_{i,\tz}\mathcal{P}(\zz,t):=\chi_{i,\tz}(z_1,t)\mathcal{P}(\zz,t), \ \ \zz=(z_1,\tz)\in \Omega^r.
\]
Moreover, by a slight abuse of notation, we define
\[
(\Theta_{t,\tz}^{-1}\mathcal{P})(\cdot,t):=\begin{pmatrix}\mathcal{P}_1(\cdot,t)-\chi_{1,\tz}\mathcal{P}_2(\cdot,t)-\dots-\chi_{r-1,\tz}\mathcal{P}_{r}(\cdot,t)\\
\mathcal{P}_2(\cdot,t)\\
\vdots\\
\mathcal{P}_r(\cdot,t)
\end{pmatrix}.
\]
We now provide an equivalent condition for weak VDMFs to be VDMFs. 

\begin{proposition}[{cf. \cite[Thm. 4.22]{PP18}}] \label{P:equiv} An element $\mathcal{P}\in \mathbb{WM}_k(\rho_m)$ is in $\mathbb{M}_k(\rho_m)$ if and only if 
\[
\Upsilon (\Theta_{t,\tz}^{-1}\mathcal{P})(\cdot,t)\in u\Hol(\Omega^{r-1},\mathbb{T})\dbl u\dbr^{r}
\]
for any choice of $\tz\in\Omega^{r-1}$.
\end{proposition}
\begin{proof} Let us set $\mathcal{P}=(\mathcal{P}_1,\dots,\mathcal{P}_r)^{\tr}\in \mathbb{M}_{k}(\rho_m)$ and write $\zz=(z_1,\tz)\in \Omega^r$. For any $a_1,\dots,a_{r-1}\in A$, consider the matrix $\gamma_{a_1,\dots,a_{r-1}}\in \GL_r(A)$ given in \eqref{E:ainv}. By the assumption on $\mathcal{P}$ and Lemma \ref{P:funceq}, we have 
\begin{multline*}
(\mathcal{P}_1-\chi_{1,\tz}\mathcal{P}_2-\dots-\chi_{r-1,\tz}\mathcal{P}_{r})(\gamma_{a_1,\dots,a_{r-1}}\cdot \zz,t)\\
=\mathcal{P}_1(\zz,t)+\sum_{i=1}^{r-1}a_i(t)\mathcal{P}_{i+1}(\zz,t)
-\chi_{1,\tz}(z_1,t)\mathcal{P}_2(\zz,t)-\cdots - \chi_{r-1,\tz}(z_1,t)\mathcal{P}_r(\zz,t)-\sum_{i=1}^{r-1}a_i(t)\mathcal{P}_{i+1}(\zz,t)\\
=\mathcal{P}_1(\zz,t)-\chi_{1,\tz}(z_1,t)\mathcal{P}_2(\zz,t)-\cdots-\chi_{r-1,\tz}(z_1,t)\mathcal{P}_{r}(\zz,t)
\end{multline*}
and for each $2\leq j \leq r$, 
\begin{equation}\label{E:funGj}
\mathcal{P}_j(\gamma_{a_1,\dots,a_{r-1}}\cdot \zz,t)=\mathcal{P}_j(\zz,t).
\end{equation}
Thus both functions $\mathcal{P}_1-\chi_{1,\tz}\mathcal{P}_2-\dots-\chi_{r-1,\tz}\mathcal{P}_{r}$ and $\mathcal{P}_j$ are $A$-invariant and hence have a $u$-expansion. Since $\mathcal{P}$ is a VDMF, we further see  that each $\mathcal{P}_j$ has a $u$-expansion with no constant term and polar part. Moreover, we  have 
\begin{equation}\label{E:lim2}
     \lim_{\substack{\zz=(z_1.\tz)\in \mathfrak{I}_{\tz}\\ |\zz|_{\infty}\to\infty}} u(\zz)\mathcal{P}_1(\zz,t)=0.
\end{equation}
Therefore, combining Lemma \ref{L:growth} and \eqref{E:lim2}, we see that the $u$-expansion of the entries of $ \Upsilon (\Theta_{t,\tz}^{-1}\mathcal{P})(\cdot,t)$ has no constant term and polar part and hence it must lie in $u\Hol(\Omega^{r-1},\mathbb{T})\dbl u \dbr^{r}$. For the other direction, if we assume that $\Upsilon (\Theta_{t,\tz}^{-1}\mathcal{P})(\cdot,t)\in u\Hol(\Omega^{r-1},\mathbb{T})\dbl u\dbr^{r}$, then it implies that for each $2\leq j \leq r$, $\mathcal{P}_j$ satisfies \eqref{E:funGj} and has a $u$-expansion with no constant term and polar part. Combining this with Lemma \ref{L:growth}, we also obtain \eqref{E:lim2}. Thus, $\mathcal{P}\in \mathbb{M}_k(\rho_m)$, finishing the proof of the proposition.
\end{proof}

We finish this subsection with a fundamental structural result on VDMFs.

\begin{theorem}\label{T:str} For any non-negative integer $k$, we have 
\[
\mathbb{M}_{k}(\rho_m)=\mathbb{M}^{m-1}_{k-\left(\frac{q^r-1}{q-1}-1\right)}\mathcal{G}_1\oplus \cdots \oplus\mathbb{M}^{m-1}_{k-\left(\frac{q^r-1}{q-1}-q^{r-1}\right)}\mathcal{G}_{r}.
\]
\end{theorem}
\begin{proof}Our method is motivated by the idea of the proof of \cite[Thm. 3.9]{PP18}. Set $ \mathrm{E}(\zz,t):=\det(\mathcal{E}(\zz,t))\mathcal{F}(\zz,t)=(\mathcal{G}_1(\zz,t),\dots,\mathcal{G}_r(\zz,t))\in \GL_r(\mathbb{T})$. Since
\[
\det(\mathrm{E}(\zz,t))=\left(\frac{\tilde{\pi}^{\frac{q^r-1}{q-1}}h_r(\zz)}{\omega(t)^{(r)}}\right)^r\frac{\omega(t)}{\tilde{\pi}^{\frac{q^r-1}{q-1}}h_r(\zz)}=\left(\frac{\tilde{\pi}^{\frac{q^r-1}{q-1}}h_r(\zz)}{\omega(t)}\right)^{r-1}\frac{1}{(t-\theta)^r\cdots(t-\theta^{q^{r-1}})^r}
\]
    and $h_r$ is a non-vanishing function on $\Omega^r$, $\det(\mathrm{E}(\zz,t))$ is a unit in $\mathbb{T}$. We consider the map $\iota: \mathbb{M}^{m-1}_{k-\left(\frac{q^r-1}{q-1}-1\right)}\oplus \cdots \oplus\mathbb{M}^{m-1}_{k-\left(\frac{q^r-1}{q-1}-q^{r-1}\right)} \to \mathbb{M}_k(\rho_m) $ given by 
    \[
    \iota(\mathcal{L}):=\mathrm{E}\mathcal{L}=\mathcal{G}_1\mathcal{L}_1+\cdots+ \mathcal{G}_{r}\mathcal{L}_{r}, \ \ \mathcal{L}=(\mathcal{L}_1,\dots,\mathcal{L}_r)^{\tr}\in \mathbb{M}^{m-1}_{k-\left(\frac{q^r-1}{q-1}-1\right)}\oplus \cdots \oplus\mathbb{M}^{m-1}_{k-\left(\frac{q^r-1}{q-1}-q^{r-1}\right)}
    \]
and claim that it is an isomorphism of $\mathbb{T}$-modules. Clearly, for any $\gamma\in \GL_r(A)$, we have
\[
\mathrm{E}(\gamma\cdot \zz,t)\mathcal{L}(\gamma\cdot \zz,t)=j(\gamma;\zz)^k\rho_m(\gamma)\mathrm{E}(\zz,t)\mathcal{L}(\zz,t).
\]
On the other hand, since $h_r$ is a Drinfeld cusp form, by Lemma \ref{L:Cor516}, for any choice of  $\tz\in \Omega^{r-1}$, we have
\[
\Upsilon\mathrm{E}(\zz,t)\mathcal{L}(\zz,t)\to \begin{pmatrix}
    0\\
    \vdots\\
    0
\end{pmatrix}
\]
as $|\zz|_{\infty}\to \infty$ where $\zz=(z_1,\tz)\in \mathfrak{I}_{\tz}$. Therefore, $\iota(\mathcal{L})\in \mathbb{M}_{k}(\rho_m)$. On the other hand, since $\mathrm{E}(\zz,t)$ is invertible, injectivity of $\iota$ immediately follows.

Now we claim that $\iota$ is surjective. First, we let $\gamma\in \GL_r(A)$ and by \eqref{E:andgenfunceq} as well as  the fact that $h_r\in \mathcal{M}_{q^r-1/(q-1)}^{1}$, we observe 
\begin{equation}\label{E:funceqF}
\mathrm{E}(\gamma\cdot \zz,t)=\rho(\gamma)\mathrm{E}(\zz,t)\begin{pmatrix}
    j(\gamma;\zz)^{\frac{q^r-1}{q-1}-1}& & & & \\
     & j(\gamma;\zz)^{\frac{q^r-1}{q-1}-q} & & & \\
     & & \ddots & & \\
     & & & \ddots & \\
     & & & & j(\gamma;\zz)^{\frac{q^r-1}{q-1}-q^{r-1}}
\end{pmatrix}.
\end{equation}
Recall from Remark \ref{R:intr} that $\Cof(\zz,t)$ is the cofactor matrix of $\mathcal{F}(\zz,t)$
and we have
\[
\mathcal{F}(\zz,t)^{-1}=\frac{\tilde{\pi}^{\frac{q^r-1}{q-1}}h_r(\zz)}{\omega(t)}\Cof(\zz,t)^{\tr}.
\]
Therefore, using \eqref{E:det2} yields
\[
\mathrm{E}(\zz,t)^{-1}=\frac{\omega(t)^{(r)}}{\tilde{\pi}^{\frac{q^r-1}{q-1}}h_r(\zz)}\frac{\tilde{\pi}^{\frac{q^r-1}{q-1}}h_r(\zz)}{\omega(t)}\Cof(\zz,t)^{\tr}
=(t-\theta)\cdots (t-\theta^{q^{r-1}})\Cof(\zz,t)^{\tr}.
\]
Let $\mathcal{P}=(\mathcal{P}_1,\dots,\mathcal{P}_r)^{\tr}\in \mathbb{M}_k(\rho_m)$. Then by using \eqref{E:funceqF} as well as our assumption on $\mathcal{P}$, we see that 
\begin{equation}\label{E:weakmod}
(\mathrm{E}^{-1}\mathcal{P})(\gamma\cdot \zz,t)=\det(\gamma)^{1-m}\begin{pmatrix}
    j(\gamma;\zz)^{k-\left(\frac{q^r-1}{q-1}-1\right)} & &  & \\
     &\ddots & & \\
     & & \ddots & \\
     & & & j(\gamma;\zz)^{k-\left(\frac{q^r-1}{q-1}-q^{r-1}\right)}
\end{pmatrix}(\mathrm{E}^{-1}\mathcal{P})(\zz,t).
\end{equation}
For $1\leq i \leq r$, let $\mathfrak{g}_i(\cdot,t)$ be the $i$-th entry of $(\mathrm{E}^{-1}\mathcal{P})(\cdot,t)$. Our goal is to show that each $\mathfrak{g}_i(\cdot,t)$ is a $\mathbb{T}$-valued Drinfeld modular form. First, since $\mathcal{P}$ is a VDMF, by Proposition \ref{P:equiv}, we have 
\begin{equation}\label{E:limitweak1}
\lim_{\substack{\zz=(z_1,\tz)\in \mathfrak{I}_{\tz}\\ |\zz|_{\infty}\to\infty}}u(\zz)\mathcal{P}_{1}(\zz,t)=0
\end{equation}
and for $2\leq \mu\leq r$, $\mathcal{P}_{\mu}\in u\mathbb{T}\dbl u\dbr$. For $1\leq i,j \leq r$, let $\mathcal{C}_{ij}(\zz,t)\in \mathbb{T}$ be the $(i,j)$-cofactor of $\mathcal{F}(\zz,t)$. Then by Lemma \ref{L:Cor516}, for any $\tz\in \Omega^{r-1}$ and $2\leq \mu \leq r$, we have
\begin{equation}\label{E:limitcofactor}
\lim_{\substack{\zz=(z_1,\tz)\in \mathfrak{I}_{\tz}\\ |\zz|_{\infty}\to\infty}}\mathcal{C}_{\mu j}(\zz,t)\mathcal{P}_{\mu}(\zz,t)<\infty
\end{equation}
and, by \eqref{E:limit11}, note that
\begin{equation}\label{E:limitweak2}
\lim_{\substack{\zz=(z_1,\tz)\in \mathfrak{I}_{\tz}\\ |\zz|_{\infty}\to\infty}}\mathcal{C}_{1 j}(\zz,t)<\infty.
\end{equation}
Observe that $\mathfrak{g}_i(\zz,t)=\sum_{i=1}^r\mathcal{C}_{ij}(\zz,t)\mathcal{P}_i(\zz,t)$. Thus, using \eqref{E:weakmod} and \eqref{E:limitweak1}--\eqref{E:limitweak2}, we see that $\mathfrak{g}_i(\cdot,t)$ is a weak $\mathbb{T}$-valued Drinfeld modular form. More precisely, we obtain
\begin{equation}\label{E:limitmod}
 \lim_{\substack{\zz=(z_1,\tz)\in \mathfrak{I}_{\tz}\\ |\zz|_{\infty}\to\infty}} u(\zz)\mathfrak{g}_i(\zz,t)=0.
\end{equation}
We claim that \eqref{E:limitmod} implies that $\mathfrak{g}_i(\cdot,t)$ is holomorphic at infinity. Indeed, if there was a negative integer $n$ so that the coefficient $\mathfrak{f}_n:\Omega^{r-1}\to \mathbb{T}$ of $u^n$ in the $u$-expansion of $\mathfrak{g}_i(\cdot,t)$ is not identically zero, then there would exist a $\tz\in \Omega^{r-1}$ so that 
 \[
 \lim_{\substack{\zz=(z_1,\tz)\in \mathfrak{I}_{\tz}\\ |\zz|_{\infty}\to\infty}} u(\zz)\mathfrak{g}_i(\zz,t)\neq 0
 \]
 which would contradict to \eqref{E:limitmod}. Thus, the $u$-expansion of $\mathfrak{g}_i(\cdot,t)$ has no polar part and hence, combining with \eqref{E:weakmod}, $\mathfrak{g}_i(\cdot,t)$ is a $\mathbb{T}$-valued Drinfeld modular form in $\mathbb{M}^{m-1}_{k-\left(\frac{q^r-1}{q-1}-q^{i-1}\right)}$. Therefore, $\iota$ is surjective and we finish the proof of our claim.

\end{proof}

\section{Hecke operators on vectorial Drinfeld modular forms} Recall that $\mathfrak{p}$ is a non-constant monic irreducible polynomial in $A$. We define the $k$-th slash operator $||_{k,m}$ with respect to type $m\in \mathbb{Z}/(q-1)\mathbb{Z}$ acting on $\Hol(\Omega^r,\TT^r)$ by 
\begin{multline*}
(\mathcal{P}||_{k,m}\gamma)(\zz):=j(\gamma;\zz)^{-k}\rho_m(\gamma)^{-1}\mathcal{P}(\gamma\cdot \zz)\\
=j(\gamma;\zz)^{-k}\det(\overline{\gamma})^{m}\overline{\gamma}^{-1}\mathcal{P}(\gamma\cdot \zz), \ \ \gamma\in \GL_r(K), \ \ \mathcal{P}\in \Hol(\Omega^r,\TT^r).
\end{multline*}
The following lemma can be easily derived from the definition of $||_{k,m}$ and the fact that (see \cite[(1.4)]{BBP18})
\[
j(\gamma\gamma';\zz)=j(\gamma;\gamma'\cdot \zz)j(\gamma';\zz) , \ \ \gamma,\gamma'\in \GL_r(K).
\]
We leave the details of its proof to the reader.
\begin{lemma}\label{L:slash} The following statements hold.
\begin{itemize}
    \item[(i)] For any $\gamma,\gamma'\in \GL_r(K)\cap \Mat_{r}(A)$ and $\mathcal{P}\in \Hol(\Omega^r,\TT^r)$, we have 
    \[
    \mathcal{P}||_{k,m}\gamma \gamma'=(\mathcal{P}||_{k,m}\gamma)||_{k,m}\gamma'.
    \]
    \item[(ii)] $\mathcal{P}$ is a weak VDMF of weight $k$ with respect to $\rho_m$ if and only if $\mathcal{P}||_{k,m}\gamma=\mathcal{P}$ for all $\gamma\in \GL_r(A)$.
\end{itemize}

\end{lemma}
 Recall $\delta\in \GL_r(K)$ from \S2.4 as well as the set of matrices $B_{\ell,r}$ so that their union $\cup_{\ell=1}^{r}B_{\ell,r}$ forms a set of coset representatives for $\GL_r(A)\setminus \GL_r(A)\delta \GL_r(A)$.
\begin{definition} We define our Hecke operator $\mathrm{T}_{\mathfrak{p},r}:\mathbb{WM}_{k}(\rho_m)\to \mathbb{WM}_{k}(\rho_m)$ by 
\[
\mathrm{T}_{\mathfrak{p},r}(\mathcal{P}):=\mathfrak{p}^k\sum_{\gamma}\mathcal{P}||_{k,m}\gamma , \ \ \mathcal{P}\in \mathbb{WM}_{k}(\rho_m)
\]
where $\gamma$ runs through elements of the union $ \cup_{\ell=1}^rB_{\ell,r}$ of matrices in $\Mat_r(A)$.
 \end{definition}
 \begin{remark} We note that the well-definedness of our Hecke operator follows from Proposition \ref{P:repr} and Lemma \ref{L:slash}.      
 \end{remark}
\begin{remark} We emphasize that, similar to Remark \ref{R:Hecke}, our notation $\mathrm{T}_{\mathfrak{p},r}$ does not include $k$ and $m$. However, since we consider $\mathrm{T}_{\mathfrak{p},r}(\mathcal{P})$ for some $\mathcal{P}\in \mathbb{WM}_k(\rho_m)$, the integers $k$ and $m$ should be clear from the weight and the corresponding representation $\rho_m$ of the weak vectorial Drinfeld modular form $\mathcal{P}$.     
\end{remark}
 
\begin{proposition}\label{P:pres} The operator $\mathrm{T}_{\mathfrak{p},r}$ is an endomorphism of the $\mathbb{T}$-module $\mathbb{M}_{k}(\rho_m)$.
\end{proposition}
\begin{proof} It suffices to show that for any $\mathcal{P}\in \mathbb{M}_{k}(\rho_m)$, $\mathrm{T}_{\mathfrak{p},r}(\mathcal{P})$ satisfies the condition in Definition \ref{D:vdmf}. Let  $\bb=(b_1,\dots,b_r)\in A^{r}$ be a tuple.  Recall that as for $\zz=(z_1, \tz)$, we define $\widetilde \bb=(b_2, \cdots, b_r)$ so that $\bb=(b_1, \widetilde \bb)$.  

We start with claiming that for any $2\leq \ell \leq r$,  $\zz=(z_1,z_2,\dots,z_r)=(z_1,\tz)\in \mathfrak{I}_{\tz}$ if and only if $\beta_{\ell,(0,\widetilde \bb)}\cdot \zz\in \mathfrak{I}_{\beta_{\ell-1,\widetilde \bb}\cdot \tz}$. Indeed,  note that, when $2\leq \ell < r$, our claim simply follows from the fact that the free $K_{\infty}$-module of rank $r-1$ generated by the entries of $ \beta_{\ell-1,\widetilde \bb}\cdot \tz $ is also the free $K_{\infty}$-module of the same rank generated by the entries of $\tz$. On the other hand, since the free $K_{\infty}$-module of rank $r-1$ generated by the entries of $ \beta_{r,\widetilde\bb}\cdot \tz $ is also the free $K_{\infty}$-module generated by $\frac{z_2}{\mathfrak{p}},\dots,\frac{z_{r-1}}{\mathfrak{p}},1$, we also obtain the claim for the case $\ell=r$. 

Next, we let $A^{r-1}\beta_{\ell,\widetilde \bb}\cdot \tz$ be the $A$-lattice of rank $r-1$ generated by the entries of $ \beta_{\ell,\widetilde \bb}\cdot \tz $. In what follows, we analyze $u(\beta_{\ell,\mathbf{b}}\cdot \zz)$.

\textbf{The case $1\leq \ell \leq r-1$: }We set 
\[
\mathfrak{u}_{\ell}:=\frac{1}{\exp_{A^{r-1}\beta_{\ell-1,\widetilde \bb}\cdot \tz}(z_1)}\ \ \ \text{ and } \ \ \ \
\widetilde{\mathfrak{e}}_{\ell}:=\exp_{A^{r-1}\beta_{\ell-1,\widetilde \bb}\cdot \tz}(b_1z_{\ell})
\]
for $b_1\in A$ such that $\deg_{\theta}(b_1)<\deg_{\theta}(\mathfrak{p})$.  Note that $\beta_{\ell,\bb}\cdot \zz =(z_1+b_1z_{\ell},\beta_{\ell-1,\widetilde \bb}\cdot \tz)$. Then we see that 
\begin{multline}\label{E:claim2}
u(\beta_{\ell,\bb}\cdot \zz )=\frac{\tilde{\pi}^{-1}}{\exp_{A^{r-1}\beta_{\ell-1,\widetilde \bb}\cdot \tz}(z_1+b_1z_{\ell})}\\= \frac{\tilde{\pi}^{-1}}{\exp_{A^{r-1}\beta_{\ell-1,\widetilde \bb}\cdot \tz}(z_1)+\exp_{A^{r-1}\beta_{\ell-1,\widetilde \bb}\cdot \tz}(b_1z_{\ell}) } 
=\frac{\tilde{\pi}^{-1}}{\mathfrak{u}_{\ell}^{-1}+\widetilde{\mathfrak{e}}_{\ell}}=\frac{\tilde{\pi}^{-1}\mathfrak{u}_{\ell}}{1+\mathfrak{u}_{\ell}\widetilde{\mathfrak{e}}_{\ell}}\\
=\tilde{\pi}^{-1}\mathfrak{u}_{\ell}\sum_{n=0}^{\infty}(-1)^n \widetilde{\mathfrak{e}}_{\ell}^n\mathfrak{u}_{\ell}^n
\end{multline} 
where the last equality is well-defined for sufficiently large $z_1$. 

\textbf{The case $\ell=r$: } In this case, our analysis follows the same argument above up to certain changes. Namely, we set 
\[
\mathfrak{u}_{r}:=\frac{1}{\exp_{A^{r-1}\beta_{r-1,\widetilde \bb}\cdot \tz}(z_1/\mathfrak{p})}\ \ \ \text{ and } \ \ \ \
\widetilde{\mathfrak{e}}_{r}:=\exp_{A^{r-1}\beta_{r-1,\widetilde \bb}\cdot \tz}(b_1/\mathfrak{p})
\]
for $b_1\in A$ such that $\deg_{\theta}(b_1)<\deg_{\theta}(\mathfrak{p})$.  In this case, we have $\beta_{\ell,\bb}\cdot \zz =\left(\frac{z_1+b_1}{\mathfrak{p}},\beta_{r-1,\widetilde \bb}\cdot \tz\right)$. Then, similar to \eqref{E:claim2}, we see that
\begin{equation}\label{E:claim3}
u(\beta_{r,\bb}\cdot \zz )
=\tilde{\pi}^{-1}\mathfrak{u}_{r}\sum_{n=0}^{\infty}(-1)^n \widetilde{\mathfrak{e}}_{r}^n\mathfrak{u}_{r}^n
\end{equation} 
whenever $z_1$ is sufficiently large.

After we establish our claim as well as obtain \eqref{E:claim2} and \eqref{E:claim3}, inspired by the idea of the proof of \cite[Prop. 5.12]{PP18}, we now use Proposition \ref{P:equiv} to finish the proof of the proposition. We can write 
\[
\mathcal{P}(\zz)=\Theta_{t,\tz}(z_1)\begin{pmatrix}
    \mathfrak{h}_1(\zz)\\
    \vdots\\
    \mathfrak{h}_r(\zz)
\end{pmatrix}=\begin{pmatrix}\mathfrak{h}_1(\zz)+\chi_{1,\tz}(z_1)\mathfrak{h}_2(\zz)+\cdots+\chi_{r-1,\tz}(z_1)\mathfrak{h}_r(\zz)\\
\mathfrak{h}_2(\zz)\\
\vdots\\
\mathfrak{h}_r(\zz)
\end{pmatrix}
\] for some $u\mathfrak{h}_1,\mathfrak{h}_2,\dots,\mathfrak{h}_r\in u\Hol(\Omega^{r-1},\mathbb{T})\dbl u\dbr$. Consider  
\[
\mathfrak{g}(\zz,t):=\chi_{1,\tz}(z_1,t)\mathfrak{h}_2(\zz,t)+\cdots+\chi_{r-1,\tz}(z_1,t)\mathfrak{h}_r(\zz,t).
\]
Letting $\mathcal{B}_1:=[\mathfrak{h}_1,\dots,\mathfrak{h}_r]^{\tr}$ and $\mathcal{B}_2:=[\mathfrak{g},0,\dots,0]^{\tr}$, we now set 
\[
\mathrm{T}_{\mathfrak{p},r}^{*}(\mathcal{B}_1):=\mathfrak{p}^k\sum_{\gamma}\mathcal{B}_1||_{k,m}\gamma
\]
where $\gamma$ runs through the elements of the union $ \cup_{\ell=1}^rB_{\ell,r}$ and 
\[
\mathrm{T}_{\mathfrak{p},r}^{*}(\mathcal{B}_2):=\mathfrak{p}^k\sum_{\gamma}\mathcal{B}_2||_{k,m}\gamma=\begin{pmatrix}\mathfrak{p}^k\left( \mathfrak{p}(t)^m\sum_{\tilde{\gamma}_1}\mathfrak{g}|_{k}\tilde{\gamma}_1+ \mathfrak{p}(t)^{m-1} \mathfrak{g}|_{k}\beta_{1,(\mathfrak{p},0,\dots,0)} \right)\\
0\\
\vdots\\
0
\end{pmatrix}
\]
where $\tilde{\gamma}_1$ runs through elements of $ \cup_{\ell=2}^{r}B_{\ell,r}$. Using the definition of the operator $||_{k,m}$, we further observe that 
\[
\mathrm{T}_{\mathfrak{p},r}(\mathcal{P})=\mathrm{T}_{\mathfrak{p},r}^{*}(\mathcal{B}_1)+\mathrm{T}_{\mathfrak{p},r}^{*}(\mathcal{B}_2).
\]
For any $\tz\in \Omega^{r-1}$, $\zz=(z_1,\tz) \in \mathfrak{I}_{\tz}$ and $i=1,2$, our next goal is to show that 
\begin{equation}\label{E:limits}
\Upsilon\mathrm{T}_{\mathfrak{p},r}^{*}(\mathcal{B}_1)\to \begin{pmatrix}
    0\\
    \vdots\\
    0
\end{pmatrix}   \text{ and } \mathrm{T}_{\mathfrak{p},r}^{*}(\mathcal{B}_2)\to \begin{pmatrix}
    0\\
    \vdots\\
    0
\end{pmatrix} 
\end{equation}
as $|\zz|_{\infty}\to \infty$. First, note that, since $\mathfrak{h}_1,\dots,\mathfrak{h}_r$ has a $u$-expansion, applying the same idea in the proof of Theorem \ref{T:HeckeDr}, we obtain that each entry of $\mathrm{T}_{\mathfrak{p},r}^{*}(\mathcal{B}_1)$ has a $u$-expansion. Furthermore, if we use the assumption $\mathfrak{h}_2,\dots,\mathfrak{h}_r\in u\Hol(\Omega^{r-1},\mathbb{T})\dbl u\dbr$, we see that each entry of $\Upsilon \mathrm{T}_{\mathfrak{p},r}^{*}(\mathcal{B}_1)$ has a $u$-expansion with no constant term. Hence, by \eqref{E:ulimit}, it finishes the proof of  the first statement in \eqref{E:limits}.  

Now we show the second statement in \eqref{E:limits}. We divide our argument into two cases. More precisely, for each $1\leq \ell \leq r$, we will analyze the terms $\mathfrak{g}(\tilde{\gamma}_1\cdot \zz,t)$ and $ \mathfrak{g}(\beta_{1,(\mathfrak{p},0,\dots,0)}\cdot \zz,t)$ where $\tilde{\gamma}_1$ runs through elements of $ \cup_{\ell=2}^{r}B_{\ell,r}$ and  $\beta_{1,(\mathfrak{p},0,\dots,0)}$ is the only element in $B_{1,r}$.

\textbf{The case $1\leq \ell \leq r-1$: } Using the fact that $\beta_{\ell,\bb}\cdot \zz =(z_1+b_1z_{\ell},\beta_{\ell-1,\widetilde \bb}\cdot \tz)$ and the definition of $\mathfrak{g}(\cdot,t)$, for $1\leq \mu \leq r-1$ and $1\leq \ell \leq r-1$, it suffices to show that
\begin{equation}\label{E:claim4}
   \lim_{\substack{\zz=(z_1,\tz)\in \mathfrak{I}_{\tz}\\ |\zz|_{\infty}\to\infty}} \chi_{\mu, \beta_{\ell-1,\widetilde{\bb}}\cdot \tz }(z_1+b_1z_{\ell},t)\mathfrak{h}_{\mu+1}(\beta_{\ell,\bb}\cdot \zz,t) =0.
   \end{equation}
  First we let $\ell=1$. Then \eqref{E:claim4} becomes
\[
 \lim_{\substack{\zz=(z_1,\tz)\in \mathfrak{I}_{\tz}\\ |\zz|_{\infty}\to\infty}} \chi_{\mu,\tz}(\mathfrak{p}z_1,t)\mathfrak{h}_{\mu+1}(\beta_{1,(\mathfrak{p},0,\dots,0)}\cdot \zz,t) =0.
 \]
Now recall $u_{\mathfrak{p}}(\zz)$ defined in \eqref{E:defua}. Since each $\mathfrak{h}_2,\dots,\mathfrak{h}_r$ has a $u$-expansion with no constant term, our case reduces to show that 
\begin{equation}\label{E:claimell0}
 \lim_{\substack{\zz=(z_1,\tz)\in \mathfrak{I}_{\tz}\\ |\zz|_{\infty}\to\infty}} \chi_{\mu,\tz}(\mathfrak{p}z_1,t)u_{\mathfrak{p}}(\zz) =0.
 \end{equation}
Since $ \zz=(z_1,\tz)\in \mathfrak{I}_{\tz}$ if and only if $ \zz=(\mathfrak{p}z_1,\tz)\in \mathfrak{I}_{\tz} $, \eqref{E:claimell0} follows from Lemma \ref{L:growth}. Next, we let $2\leq \ell \leq r-1$. In this case, \eqref{E:claim4} becomes
\begin{multline*}
 \lim_{\substack{\zz=(z_1,\tz)\in \mathfrak{I}_{\tz}\\ |\zz|_{\infty}\to\infty}} \chi_{\mu,\beta_{\ell-1,\widetilde \bb}\cdot \tz }(z_1+b_1z_{\ell},t)\mathfrak{h}_{\mu+1}(\beta_{\ell,\bb}\cdot \zz,t) \\=
\lim_{\substack{\zz=(z_1,\tz)\in \mathfrak{I}_{\tz}\\ |\zz|_{\infty}\to\infty}} \chi_{\mu,\beta_{\ell-1,\widetilde\bb}\cdot \tz }(z_1,t)\mathfrak{h}_{\mu+1}(\beta_{\ell,\bb}\cdot \zz,t) \\
+\lim_{\substack{\zz=(z_1,\tz)\in \mathfrak{I}_{\tz}\\ |\zz|_{\infty}\to\infty}}\chi_{\mu,\beta_{\ell-1,\widetilde \bb}\cdot \tz}( b_1z_{\ell},t)\mathfrak{h}_{\mu+1}(\beta_{\ell,\bb}\cdot \zz,t)
= 0.
 \end{multline*}
Again, since each $\mathfrak{h}_2,\dots,\mathfrak{h}_r$ has a $u$-expansion with no constant term and $ \chi_{\mu,\beta_{\ell-1,\widetilde \bb}\cdot \tz }( b_1z_{\ell},t) $ is independent of $z_1$, it suffices to show that 
\begin{equation}\label{E:claimellgeneral}
 \lim_{\substack{\zz=(z_1,\tz)\in \mathfrak{I}_{\tz}\\ |\zz|_{\infty}\to\infty}} \chi_{\mu,\beta_{\ell-1,\widetilde \bb}\cdot \tz }(z_1,t)u(\beta_{\ell,\bb}\cdot \zz ) =0.
 \end{equation}
Finally, using our first claim at the beginning of the proof and \eqref{E:claim2}, we see that \eqref{E:claimellgeneral} reduces to analyze the limit
\[
 \lim_{\substack{(z_1,\beta_{\ell-1,\widetilde\bb}\cdot \tz)\in \mathfrak{I}_{\beta_{\ell-1,\widetilde \bb}\cdot \tz}\\ |(z_1,\beta_{\ell-1,\widetilde \bb}\cdot \tz)|_{\infty}\to\infty}} \chi_{\mu,\beta_{\ell-1,\widetilde\bb}\cdot \tz }(z_1,t)\mathfrak{u}_{\ell} = \lim_{\substack{(z_1,\beta_{\ell-1,\widetilde\bb}\cdot \tz)\in \mathfrak{I}_{\beta_{\ell-1,\widetilde \bb}\cdot \tz}\\ |(z_1,\beta_{\ell-1,\widetilde \bb}\cdot \tz)|_{\infty}\to\infty}} \chi_{\mu,\beta_{\ell-1,\widetilde\bb}\cdot \tz }(z_1,t)\exp_{A^{r-1}\beta_{\ell-1,\widetilde \bb}\cdot \tz}(z_1)^{-1}.
\]
Replacing $\tz$ with $\beta_{\ell-1,\widetilde\bb}\cdot \tz  $ in Lemma \ref{L:growth}, we finally obtain
\[
\lim_{\substack{(z_1,\beta_{\ell-1,\widetilde\bb}\cdot \tz)\in \mathfrak{I}_{\beta_{\ell-1,\widetilde \bb}\cdot \tz}\\ |(z_1,\beta_{\ell-1,\widetilde \bb}\cdot \tz)|_{\infty}\to\infty}} \chi_{\mu,\beta_{\ell-1,\widetilde\bb}\cdot \tz }(z_1,t)\exp_{A^{r-1}\beta_{\ell-1,\widetilde \bb}\cdot \tz}(z_1)^{-1}=0,
\] 
implying \eqref{E:claimellgeneral}.

\textbf{The case $\ell=r$: } Note that $\beta_{r,\bb}\cdot \zz =\left(\frac{z_1+b_1}{\mathfrak{p}},\beta_{r-1,\widetilde \bb}\cdot \tz\right)$. As in the previous case, it is enough to show that
\begin{equation}\label{E:claim5}
   \lim_{\substack{\zz=(z_1,\tz)\in \mathfrak{I}_{\tz}\\ |\zz|_{\infty}\to\infty}} \chi_{\mu, \beta_{r-1,\widetilde{\bb}}\cdot \tz }\left(\frac{z_1+b_1}{\mathfrak{p}},t\right)\mathfrak{h}_{\mu+1}(\beta_{r,\bb}\cdot \zz,t) =0.
   \end{equation}
Proceeding as above, \eqref{E:claim5} reduces to show that 
\begin{equation}\label{E:claimellgeneral2}
 \lim_{\substack{\zz=(z_1,\tz)\in \mathfrak{I}_{\tz}\\ |\zz|_{\infty}\to\infty}} \chi_{\mu, \beta_{r-1,\widetilde{\bb}}\cdot \tz}\left(z_1/\mathfrak{p},t\right)u(\beta_{r,\bb}\cdot \zz ) =0.
 \end{equation}
Again using our first claim at the beginning of the proof and \eqref{E:claim3}, we see that \eqref{E:claimellgeneral2} reduces to analyze the limit
\begin{multline*}
 \lim_{\substack{(z_1,\beta_{r-1,\widetilde\bb}\cdot \tz)\in \mathfrak{I}_{\beta_{r-1,\widetilde \bb}\cdot \tz}\\ |(z_1,\beta_{r-1,\widetilde \bb}\cdot \tz)|_{\infty}\to\infty}} \chi_{\mu,\beta_{r-1,\widetilde\bb}\cdot \tz }\left(z_1/\mathfrak{p},t\right)\mathfrak{u}_{r} \\= \lim_{\substack{(z_1,\beta_{r-1,\widetilde\bb}\cdot \tz)\in \mathfrak{I}_{\beta_{r-1,\widetilde \bb}\cdot \tz}\\ |(z_1/\mathfrak{p},\beta_{\ell-1,\widetilde \bb}\cdot \tz)|_{\infty}\to\infty}} \chi_{\mu,\beta_{r-1,\widetilde\bb}\cdot \tz }(z_1/\mathfrak{p},t)\exp_{A^{r-1}\beta_{r-1,\widetilde \bb}\cdot \tz}(z_1/\mathfrak{p})^{-1}.
\end{multline*}
Since $ (z_1,\beta_{r-1,\widetilde \bb}\cdot \tz)\in \mathfrak{I}_{\beta_{r-1,\widetilde \bb}\cdot \tz}$ if and only if $ (z_1/\mathfrak{p},\beta_{r-1,\widetilde \bb}\cdot \tz)\in \mathfrak{I}_{\beta_{r-1,\widetilde \bb}\cdot \tz}$, replacing $\tz$ with $\beta_{r-1,\widetilde\bb}\cdot \tz  $ in Lemma \ref{L:growth}, we obtain
\[
\lim_{\substack{(z_1,\beta_{r-1,\widetilde\bb}\cdot \tz)\in \mathfrak{I}_{\beta_{r-1,\widetilde \bb}\cdot \tz}\\ |(z_1/\mathfrak{p},\beta_{\ell-1,\widetilde \bb}\cdot \tz)|_{\infty}\to\infty}} \chi_{\mu,\beta_{r-1,\widetilde\bb}\cdot \tz }(z_1/\mathfrak{p},t)\exp_{A^{r-1}\beta_{r-1,\widetilde \bb}\cdot \tz}(z_1/\mathfrak{p})^{-1}=0,
\] 
implying \eqref{E:claimellgeneral2}. Thus, we show \eqref{E:limits} which implies that $\mathrm{T}_{\mathfrak{p},r}(\mathcal{P})$ satisfies the condition in Definition \ref{D:vdmf}. That is, $\mathrm{T}_{\mathfrak{p},r}(\mathcal{P})\in \mathbb{M}_{k}(\rho_m)$ as desired.
\end{proof}
Let $\mathcal{T}_{\mathfrak{p},r,k,m}$ be the operator given by the last coordinate of $\mathrm{T}_{\mathfrak{p},r}$ which acts on the space $ \mathbb{M}_{k}(\rho_m)$. Our next lemma simply follows from the definition of $\mathrm{T}_{\mathfrak{p},r}$ and the structure of elements in $B_{\ell,r}$ for each $1\leq \ell \leq r$ which are all upper triangular matrices.
\begin{lemma}\label{L:Heckelast} Let $\mathcal{P}=(\mathcal{P}_1,\dots,\mathcal{P}_r)^{\tr}\in \mathbb{M}_{k}(\rho_m)$. Then we have
\[
\mathcal{T}_{\mathfrak{p},r,k,m}(\mathcal{P}_r)=\mathfrak{p}^k\left( \mathfrak{p}(t)^m\sum_{\gamma_1}\mathcal{P}_r|_{k}\gamma_1+\mathfrak{p}(t)^{m-1}\sum_{\gamma_2} \mathcal{P}_r|_{k}\gamma_2 \right)
\]
where $\gamma_1$ runs through elements of $ \cup_{\ell=1}^{r-1}B_{\ell,r}$ and $\gamma_2$ runs through elements of $B_{r,r}$. 
\end{lemma}
The following immediate corollary of Lemma \ref{L:Heckelast} is crucial to prove our main result in this section.

\begin{corollary}\label{C:HeckeRed} Let $\mathcal{P}=(\mathcal{P}_1,\dots,\mathcal{P}_r)^{\tr}\in \mathbb{M}_{k}(\rho)$ and assume that $\mathcal{P}_r(\zz,\theta^{q^n})\in \mathcal{M}_{k+q^{n}}$. Then, upon the substitution $t=\theta^{q^{n}}$,  $\mathcal{T}_{\mathfrak{p},r,k,1}$ induces the Hecke operator $T_{\mathfrak{p},r}$ acting on $\mathcal{M}_{k+q^{n}}$. 
\end{corollary}

In what follows, we analyze the last entry of an element in $\mathbb{M}_{k}(\rho_m)$. Note that if $\mathcal{P}=(\mathcal{P}_1,\dots,\mathcal{P}_r)^{\tr}\in \mathbb{M}_{k}(\rho)$ then $\mathcal{P}_r$ is $A$-invariant. Moreover, by the condition in Definition \ref{D:vdmf}, $\mathcal{P}_r$ has a $u$-expansion with no polar part.

For each $2\leq \ell \leq r$ and $\bb_{\ell}\in A^{r}$, let $L_{\ell,\mathbf{b}_{\ell}}$ be the $\mathbb{F}_q$-vector space given in \S2.2. Recall that for $n\geq 0$, we denote by $G_{n,L_{\ell,\mathbf{b}_{\ell}}}(X)\in \mathbb{C}_{\infty}[X]$  the $n$-th Goss polynomial corresponding to $L_{\ell,\mathbf{b}_{\ell}}$. Our next proposition can be shown by using the same argument to prove Theorem \ref{T:HeckeDr}. We leave the details to the reader.

\begin{proposition}\label{P:uexp} Let $\mathcal{P}=(\mathcal{P}_1,\dots,\mathcal{P}_r)^{\tr}\in \mathbb{M}_{k}(\rho_m)$ and $\mathcal{P}_r(\cdot,t)=\sum_{n=0}^{\infty}\mathfrak{f}_n(\cdot,t)u^n$ be the $u$-expansion of $\mathcal{P}_r(\cdot,t)$. 
Then we have 
\begin{multline*}
\mathcal{T}_{\mathfrak{p},r,k,m}(\mathcal{P}_r)(\zz,t)=\mathfrak{p}^k\mathfrak{p}(t)^m\sum_{n=0}^{\infty}\mathfrak{f}_n(\tz,t)u_{\mathfrak{p}}(\zz)^n\\
+\mathfrak{p}^k\mathfrak{p}(t)^m\sum_{n=0}^{\infty}\sum_{\substack{2\leq \ell \leq r-1 \\ \beta_{\ell,\mathbf{b}_{\ell}}\in B_{\ell,r}}}(\mathfrak{f}_n|_{k-n}\widetilde{\beta_{\ell,\mathbf{b}_{\ell}}})( \tz,t)G_{n,L_{\ell,\mathbf{b}_{\ell}}}(u(\zz))\\
+\mathfrak{p}^k\mathfrak{p}(t)^{m-1}\sum_{n=0}^{\infty}\sum_{\beta_{r,\mathbf{b}_r}\in B_{r,r}}(\mathfrak{f}_n|_{k-n}\widetilde{\beta_{r,\mathbf{b}_r}})( \tz,t)G_{n,L_{r,\mathbf{b}_{r}}}(u(\zz))\end{multline*}
for any $\zz\in \Omega^{r}$ lying in some neighborhood of infinity. In particular, the constant term of the $u$-expansion of $\mathcal{T}_{\mathfrak{p},r,k,m}(\mathcal{P}_r)$ is $\mathcal{T}_{\mathfrak{p},r-1,k,m}(\mathfrak{f}_0)$ and the linear term of the $u$-expansion of $\mathcal{T}_{\mathfrak{p},r,k,m}(\mathcal{P}_r)$ is equal to $\mathfrak{p}\mathcal{T}_{\mathfrak{p},r-1,k-1,m}(\mathfrak{f}_1)$.
\end{proposition}

We are now ready to prove the main theorem of this section.

\begin{theorem}\label{T:Hecke} The function $\mathcal{G}_{r}(\cdot,t)\in \mathbb{M}_{\frac{q^{r-1}-1}{q-1}}(\rho)$ is a Hecke eigenform. Moreover, for any $\zz\in \Omega^r$, we have
\[
\mathrm{T}_{\mathfrak{p},r}(\mathcal{G}_{r})(\zz,t)=\mathfrak{p}^{1+q+\dots+q^{r-2}}\mathcal{G}_{r}(\zz,t).
\]
\end{theorem}
\begin{proof} We proceed by induction. When $r=2$, the theorem is proved in \cite[Prop. 5.16]{PP18}. Assume that the theorem holds for $r-1$. Observe that \cite[Thm. 11.1(b), Thm. 17.5(a)]{BBP18} implies  $\mathcal{M}_k^{0}=\{0\}$ if $k<0$ and $\mathcal{M}_{0}^{0}=\mathbb{C}_{\infty}$. Thus, by Corollary \ref{C:tvalmodform}, we have $\mathbb{M}_k^{0}=\{0\}$ if $k<0$ and $\mathbb{M}_{0}^{0}=\mathbb{T}$. Hence, by Theorem \ref{T:str}, we see that $\mathbb{M}_{\frac{q^{r-1}-1}{q-1}}(\rho)$ is generated by $\mathcal{G}_{r}(\cdot,t)$ over $\mathbb{T}$. On the other hand, by Proposition \ref{P:pres}, there exists $\alpha\in \mathbb{T}$ such that $\mathrm{T}_{\mathfrak{p},r}(\mathcal{G}_{r})(\zz,t)=\alpha \mathcal{G}_{r}(\zz,t)$ for all $\zz \in \Omega^r$. Therefore, to find $\alpha$, it suffices to analyze the behavior of the last entry of $\mathcal{G}_{r}(\cdot,t)$ under the Hecke operator. Since the last entry of $\mathcal{G}_{r-1}(\cdot,t)$ is $\mathcal{H}_{r-1}(\cdot,t)$, observe that, by the induction hypothesis, we have
\[
\mathcal{T}_{\mathfrak{p},r-1,\frac{q^{r-2}-1}{q-1},1}(\mathcal{H}_{r-1})(\tz,t)=\mathfrak{p}^{1+q+\cdots+q^{r-3}}\mathcal{H}_{r-1}(\tz,t).
\]
Thus, by Lemma \ref{L:uexp} and Proposition \ref{P:uexp}, analyzing the first coefficient of the $u$-expansion of $\mathrm{T}_{\mathfrak{p},r}(\mathcal{G}_{r})(\zz,t)$, we see that 
\begin{multline*}
\alpha(-1)^r\tilde{\pi}\mathcal{H}^{(1)}_{r-1}(\tz,t)=\mathfrak{p}\mathcal{T}_{\mathfrak{p},r-1,\frac{q^{r-1}-q}{q-1},1}((-1)^r\tilde{\pi}\mathcal{H}^{(1)}_{r-1}(\tz,t))\\=(-1)^r\tilde{\pi}\mathfrak{p}(\mathcal{T}_{\mathfrak{p},r-1,\frac{q^{r-2}-1}{q-1},1}(\mathcal{H}_{r-1}(\tz,t)))^{(1)}
=(-1)^r\tilde{\pi} \mathfrak{p} (\mathfrak{p}^{1+q+\cdots+q^{r-3}})^q\mathcal{H}^{(1)}_{r-1}(\tz,t)\\=\mathfrak{p}^{1+q+\cdots+q^{r-2}}(-1)^r\tilde{\pi}\mathcal{H}^{(1)}_{r-1}(\tz,t).
\end{multline*}
Here, the second equality follows from the $\mathbb{C}_{\infty}$-linearity of the operator $\mathcal{T}_{\mathfrak{p},r,k,m}$ and the fact that 
\begin{multline*}
    \mathcal{T}_{\mathfrak{p},r-1,\frac{q^{r-1}-q}{q-1},1}(\mathcal{H}^{(1)}_{r-1}(\tz,t))=\mathfrak{p}^{\frac{q^{r-1}-q}{q-1}}\left(\mathfrak{p}(t)\sum_{\gamma_1}\mathcal{H}_{r-1}^{(1)}(\gamma_1\cdot \tz,t)+\sum_{\gamma_2}\mathcal{H}_{r-1}^{(1)}(\gamma_2\cdot \tz,t)\right)\\
    =\left(\mathfrak{p}^{\frac{q^{r-2}-1}{q-1}}\left(\mathfrak{p}(t)\sum_{\gamma_1}\mathcal{H}_{r-1}(\gamma_1\cdot \tz,t)+\sum_{\gamma_2}\mathcal{H}_{r-1}(\gamma_2\cdot \tz,t)\right)\right)^{(1)}\\= \left(\mathcal{T}_{\mathfrak{p},r-1,\frac{q^{r-2}-1}{q-1},1}(\mathcal{H}^{(1)}_{r-1}(\tz,t))\right)^{(1)}
\end{multline*}
where $\gamma_1$ runs through elements of $ \cup_{\ell=1}^{r-2}B_{\ell,r-1}$ and $\gamma_2$ runs through elements of $B_{r-1,r-1}$.
Hence, $\alpha=\mathfrak{p}^{1+q+\cdots+q^{r-2}}$, finishing the proof of the theorem.    
\end{proof}

The next corollary, which is our second main result, is an immediate consequence of Lemma \ref{L:funcg}, Corollary \ref{C:HeckeRed} and Theorem \ref{T:Hecke}.
\begin{corollary}\label{C:main2}
For $n\geq r-1$, $\mathcal{H}_r(\zz,\theta^{q^n})$ is a Hecke eigenform and
\[
T_{\mathfrak{p},r}(\mathcal{H}_r)(\zz,\theta^{q^n})=\mathfrak{p}^{1+q+\dots+q^{r-2}}\mathcal{H}_r(\zz,\theta^{q^n}).
\]
\end{corollary}

\end{document}